\documentclass{amsart} 
\usepackage{amsfonts}
\usepackage{amssymb}

\newtheorem{theorem}{Theorem}[section]

\newtheorem{corollary}[theorem]{Corollary}

\newtheorem{example}[theorem]{Example}

\newtheorem{lemma}[theorem]{Lemma}

\numberwithin{equation}{section}

\begin{document}
\title[Bottom spectrum of three-dimensional manifolds]{Bottom spectrum of
three-dimensional manifolds with scalar curvature lower bound}
\author{Ovidiu Munteanu and Jiaping Wang}

\begin{abstract}
A classical result of Cheng states that the bottom spectrum of complete
manifolds of fixed dimension and Ricci curvature lower bound achieves its
maximal value on the corresponding hyperbolic space. The paper establishes
an analogous result for three-dimensional complete manifolds with scalar
curvature lower bound subject to some necessary topological assumptions. The
rigidity issue is also addressed and a splitting theorem is obtained for
such manifolds with the maximal bottom spectrum.
\end{abstract}

\address{Department of Mathematics, University of Connecticut, Storrs, CT
06268, USA}
\email{ovidiu.munteanu@uconn.edu}
\address{School of Mathematics, University of Minnesota, Minneapolis, MN
55455, USA}
\email{jiaping@math.umn.edu}
\maketitle

\section{Introduction}

The spectrum of the Laplacian $\Delta $ of a complete manifold $M$ is a
closed subset of the half line $[0,\infty ).$ Its smallest value is called
the bottom spectrum of $M$ and denoted by $\lambda _{1}(M).$ Alternatively,
it is characterized as the optimal constant of the Poincar\'{e} inequality

\begin{equation*}
\lambda _{1}\left( M\right) \,\int_{M}\varphi ^{2}\leq \int_{M}\left\vert
\nabla \varphi \right\vert ^{2},
\end{equation*}%
for all smooth functions $\varphi $ with compact support.

A classical comparison result due to Cheng \cite{C} provides a sharp upper
bound for the bottom spectrum $\lambda_1(M)$ of an $n$-dimensional complete
manifold $M$ with its Ricci curvature $\mathrm{Ric}\geq -\left( n-1\right) K$
for some nonnegative constant $K.$ Namely,

\begin{equation*}
\lambda _{1}\left( M\right) \leq \frac{\left( n-1\right) ^{2}}{4}K.
\end{equation*}

Naturally, one may wonder if the result remains true under a scalar
curvature lower bound. The answer is obviously negative in the case $n\geq
4. $ Indeed, the product manifold $\mathbb{H}^{n-2}\times \mathbb{S}^2(r)$
of the hyperbolic space $\mathbb{H}^{n-2}$ with the sphere $\mathbb{S}^2(r)$
of radius $r$ has positive bottom spectrum, yet its scalar curvature can be
made into an arbitrary positive constant by choosing $r$ small. We are thus
led to consider the case of dimension three.

\begin{theorem}
\label{A1} Let $\left( M,g\right) $ be a three-dimensional complete
Riemannian manifold with scalar curvature $S\geq -6K$ on $M$ for some
constant $K\geq 0.$ Assume that the Ricci curvature of $M$ is bounded from
below by a constant. Then the bottom spectrum of $M$ satisfies $\lambda
_{1}\left( M\right) \leq K$ provided that either of the topological
assumptions (A) or (B) is satisfied.\newline
(A) The second homology group $H_{2}\left( M,\mathbb{Z}\right) $ of $M$
contains no spherical classes;\newline
(B) $M$ has finitely many ends and its first Betti number $b_{1}(M)<\infty .$
\end{theorem}

Recall that an end of $M$ with respect to a bounded (smooth) domain $\Omega$
is simply an unbounded component of $M\setminus \Omega,$ and $M$ is said to
have finitely many ends if the number of ends is always bounded by a fixed
constant regardless of $\Omega.$ We point out that the theorem has been
proved by the authors in \cite{MW} for the case (B) under an additional
assumption that the volume of unit balls does not decay too fast at
infinity. The following example indicates the necessity of the topological
assumptions.

\begin{example}
\label{E1} Consider the connected sum $N=\left( \mathbb{S}^{2}\times \mathbb{%
S}^{1}\right) \#\left( \mathbb{S}^{2}\times \mathbb{S}^{1}\right) $.
According to \cite{GL, SY1}, $N$ carries a metric with positive scalar
curvature. So the scalar curvature of the pull-back metric on its universal
cover $M$ is positive as well. Note that the first fundamental group of $N$
is a free group on two generators. In particular, it is non-amenable. By a
result of Brooks \cite{B}, the bottom spectrum of $M$ equipped with the
pull-back metric satisfies $\lambda _{1}(M)>0.$
\end{example}

Clearly, the estimate of Theorem \ref{A1} fails for $M.$ This failure is due
to the fact that both topological conditions (A) and (B) are violated.
Indeed, for condition (A), while the first Betti number of $M$ is $0,$ the
number of ends of $M$ must be infinite as this is the case for the free
group on two generators. For condition (B), obviously, $H_2(M, \mathbb{Z})$
contains spherical classes.

Note that the estimate in Theorem \ref{A1} is sharp for the hyperbolic space 
$\mathbb{H}^{3}.$ As such, we consider the corresponding rigidity issue and
prove the following result.

\begin{theorem}
\label{A2}Let $\left( M,g\right) $ be a complete three-dimensional manifold
with scalar curvature bounded by $S\geq -6K$ and bottom of spectrum $\lambda
_{1}\left( M\right) =K$ for some constant $K>0.$ Then either $M$ has no
finite volume end or $M=\mathbb{R}\times \mathbb{T}^{2}$, the warped product
of a line and the flat torus with 
\begin{equation*}
ds_{M}^{2}=dt^{2}+e^{2\sqrt{K}\,t}ds_{\mathbb{T}^{2}}^{2}
\end{equation*}%
provided that $H_{2}\left( M,\mathbb{Z}\right) $ contains no spherical
classes and that the Ricci curvature of $M$ is bounded from below by a
constant.
\end{theorem}

Theorem \ref{A2} is very much motivated by the work of Li and the second
author \cite{LW1, LW2}, where they have addressed the similar rigidity
question for the aforementioned result of Cheng.

\begin{theorem}[Li-Wang]
\label{LW} Suppose $(M^{n},g)$ is complete, $n\geq 3,$ with $\lambda
_{1}\geq \frac{(n-1)^{2}}{4}\,K$ and $\mathrm{Ric}\geq -(n-1)K.$ Then either 
$M$ is connected at infinity or $M^{n}=\mathbb{R}\times N^{n-1}$, for some
compact $N$ with $ds_{M}^{2}=dt^{2}+e^{2\sqrt{K}\,t}\,ds_{N}^{2}$ for $n\geq
3$ or $ds_{M}^{2}=dt^{2}+\cosh ^{2}(\sqrt{K}\,t)\,ds_{N}^{2}$ when $n=3.$
\end{theorem}

The result says that $M$ must split off a line in an explicit manner if it
is not connected at infinity. It may also be viewed as a suitable
generalization of the famous Cheeger-Gromoll splitting theorem \cite{CG} for
manifolds with non-negative Ricci curvature. In comparison, Theorem \ref{A2}
is considerably weaker as it fails to say anything when all ends are of
infinite volume. Surprisingly, this happens for good reasons. In fact, the
following example shows that even within the class of warped product
manifolds the metrics with $S\geq -6$ and $\lambda_{1}\left( M\right) =1$
are not unique. Consequently, a strict analogue of Theorem \ref{LW} is not
to be expected.

\begin{example}
\label{E2} Let $M=\mathbb{R}\times \mathbb{T}^{2}$ with $ds_{M}^{2}=dt^{2}+%
\cosh ^{\frac{2}{a}}\left( at\right) ds_{\mathbb{T}^{2}}^{2}$, where $1\leq
a\leq \frac{3}{2}.$ Then the scalar curvature of $M$ satisfies 
\begin{equation*}
S=-6+\left( 6-4a\right) \cosh ^{-2}\left( at\right) \geq -6
\end{equation*}%
and $\lambda _{1}\left( M\right) =1$ as the positive function $u\left(
t\right) =\cosh ^{-\frac{1}{a}}\left( at\right) $ satisfies 
\begin{equation*}
\Delta u=\left( a-1\right) \sinh ^{2}\left( at\right) \cosh ^{-\left( 2+%
\frac{1}{a}\right) }\left( at\right) -a\cosh ^{-\frac{1}{a}}\left( at\right)
\leq -u.
\end{equation*}
\end{example}

Among this family of manifolds, only the one corresponding to $a=1$ appears
in rigidity result Theorem \ref{LW} as all others violate the Ricci lower
bound. Note also that $M$ has constant scalar curvature $S=-6$ for $a=\frac{3%
}{2}.$

We now indicate some of the ideas involved in the proof of Theorem \ref{A1}.
To do so, let us first quickly recall the argument in \cite{MW}. Assume by
contradiction that $\lambda _{1}(M)>K.$ Then $M$ is necessarily nonparabolic
and admits positive Green's functions. The assumption on the volume of unit
balls is used to show that the minimal positive Green's function $G(p,x)$ of 
$M$ with fixed $p$ goes to zero at infinity. Moreover, the topological
assumption (B) ensures the regular level sets $l(t)$ of $G$ are necessarily
connected compact surfaces on each end. The proof proceeds as in \cite{M} by
taking $\varphi =\psi \,\left\vert \nabla G\right\vert ^{1/2}$ as a test
function in the Poincar\'{e} inequality with the cut-off function $\psi $
chosen to depend on the function $G$ as well. One is then led to integrate
the following Bochner formula over the level sets $l(t).$

\begin{equation*}
\Delta \left\vert \nabla G\right\vert =\left( \left\vert G_{ij}\right\vert
^{2}-\left\vert \nabla \left\vert \nabla G\right\vert \right\vert
^{2}\right) \left\vert \nabla G\right\vert ^{-1}+\mathrm{Ric}\left( \nabla
G,\nabla G\right) \left\vert \nabla G\right\vert ^{-1}.
\end{equation*}%
The Ricci curvature term is handled by first rewriting it into

\begin{equation}
\mathrm{Ric}\left( \nabla G,\nabla G\right) \left\vert \nabla G\right\vert
^{-2}=\frac{1}{2}S-\frac{1}{2}S_{l\left( t\right) }+\frac{1}{\left\vert
\nabla G\right\vert ^{2}}\left( \left\vert \nabla \left\vert \nabla
G\right\vert \right\vert ^{2}-\frac{1}{2}\left\vert \nabla
^{2}G\right\vert^{2}\right)  \label{1}
\end{equation}%
and then applying the Gauss-Bonnet theorem, where $S_{l\left( t\right) }$ is
the scalar curvature of the level set $l\left(t\right).$

The idea of rewriting the Ricci curvature term as (\ref{1}) has origin in
Schoen and Yau \cite{SY}, where it was realized that on a minimal surface $N$
in a three-dimensional manifold $M,$

\begin{equation}
\mathrm{Ric}\left( \nu,\nu\right)=\frac{1}{2}S-\frac{1}{2}S_{N }-\frac{1}{2}%
\left\vert A\right\vert^2  \label{2}
\end{equation}
with $\nu,$ $S_{N}$ and $A$ being the unit normal vector, the scalar
curvature and the second fundamental form of $N,$ respectively. An identity
of the nature (\ref{2}) was derived subsequently for any surface $N,$ not
necessarily minimal, by Jezierski and Kijowski in \cite{JK}. Recently, this
kind of identity was discovered by Stern \cite{S} as well for the level sets
of harmonic functions.  Level set methods have been recently applied
in various  settings in  \cite{AMO, BKKS, HKK, HKKZ}. 

Now for the proof of Theorem \ref{A1}, without the volume lower bound
assumption on the unit balls, the properness of the minimal positive Green's
function is no longer guaranteed. As a result, one runs into the difficulty
that the Gauss-Bonnet formula can not be directly applied for the level sets 
$l(t)$ as they may be non-compact. To overcome this difficulty, we modify
the function $G$ by considering $u=G\,\psi,$ where $\psi$ is a smooth
cut-off function. While this guarantees that the positive level sets of $u$
are compact, the price we pay is that the function $u$ is no longer
harmonic. This creates many new technical issues. For one, it becomes
unclear whether the intersection of level sets of $u$ with each end is
connected. To get around the issue, we consider only the component $L(t,
\infty)$ of the super level set $\{u>t\}$ with the fixed pole $p\in L(t,
\infty).$ The Gauss-Bonnet formula is then applied to the compact connected
surfaces obtained from the intersection of the boundary of $L(t, \infty)$
with each unbounded component of $M\setminus L(t, \infty).$ Another issue is
that the Bochner formula for $u$ introduces many extra terms. Fortunately,
those terms can be controlled with the help of the exponential decay
estimate for the Green's function from \cite{LW1} together with a judicious
choice of the cut-off function $\psi$ based on a result from Schoen and Yau 
\cite{SY94}. Actually, instead of $G,$ we work with the so-called barrier
function $f,$ namely, a harmonic function defined on $M\setminus D$ for an
arbitrarily large compact smooth domain $D$ such that $f=1$ on the boundary $%
\partial D$ of $D$ and $0<f<1.$ The advantage is that it leads to a slightly
stronger conclusion, that is, the smallest essential spectrum of $M$ is
bounded from above by $K$ as well.

The proof of Theorem \ref{A2} follows more or less from that of Theorem \ref%
{A1} with a careful tracking of the constants. The advantage here is that we
work with a globally defined positive harmonic function. So no interior
boundary terms arise from the estimates. However, at one point, we do need
to derive precise estimates to balance the boundary terms at the infinity of
two ends.

Finally, we mention a splitting result for three dimensional manifolds with
non-negative scalar curvature. It may be viewed as the limiting case of
Theorem \ref{A2} by letting $K$ go to zero.

\begin{theorem}
\label{A3} Let $\left( M,g\right) $ be a complete three-dimensional
parabolic manifold with its scalar curvature $S\geq 0$ and Ricci curvature
bounded below by a constant. Assume that $H_{2}\left( M,\mathbb{Z}\right) $
contains no spherical classes. Then either $M$ is connected at infinity or
it splits isometrically as a direct product $M=\mathbb{R}\times \mathbb{T}%
^2. $
\end{theorem}

The proof relies on a monotonicity formula for harmonic functions originated
in \cite{MW} and refined in \cite{MW1}. In fact, a similar argument works
for the case $M$ is nonparabolic as well provided that each nonparabolic end
carries a proper barrier function.

The paper is organized as follows. In Section \ref{Est} we prove the bottom
spectrum estimate Theorem \ref{A1}. In Section \ref{Max} we present the
proof of Theorem \ref{A2}. Section \ref{Sge0} concerns splitting results for
three-manifolds with non-negative scalar curvature.\\

\textbf{Acknowledgment:} We wish to thank Otis Chodosh for his interest in
our work and for sharing Example \ref{E1} with us. We would also like to
thank Florian Johne for his careful reading of the previous version and for
his helpful comments. The first author was partially supported by NSF grant
DMS-1811845 and by a Simons Foundation grant.

\section{Spectral estimate\label{Est}}

This section is devoted to the proof of Theorem \ref{A1}. We first make some
preparations without making any dimension restriction on manifold $M.$
Denote by $r\left( x\right) $ the distance to a fixed point $p\in M.$ From
Theorem 4.2 in \cite{SY94}, we have a smooth distance-like function on an
arbitrary dimensional manifold $M.$

\begin{lemma}
\label{r}Let $\left( M,g\right) $ be an $n$-dimensional complete Riemannian
manifold with Ricci curvature bounded from below by a constant. Then there
exists a proper function $\rho \in C^{\infty }\left( M\right) $ such that%
\begin{gather}
\frac{1}{C}r\leq \rho \leq Cr  \label{c1} \\
\left\vert \nabla \rho \right\vert +\left\vert \Delta \rho \right\vert \leq
C\ \   \notag
\end{gather}%
and 
\begin{equation}
\left\vert \nabla \left( \Delta \rho \right) \right\vert \leq C\left\vert
\nabla ^{2}\rho \right\vert ,  \label{c2}
\end{equation}%
for some constant $C>0.$
\end{lemma}

Although (\ref{c2}) is not explicitly stated in \cite{SY94}, it follows
immediately from the construction of $\rho $. For given $R>0,$ let $\psi
:\left( 0,\infty \right) \rightarrow \mathbb{R}$ be a smooth function such
that $\psi =1$ on $\left( 0,R\right) $ and $\psi =0$ on $\left( 2R,\infty
\right) $ satisfying

\begin{equation}
\left\vert \psi ^{\prime }\right\vert \leq \frac{C}{R},\ \ \left\vert \psi
^{\prime \prime }\right\vert \leq \frac{C}{R^{2}}\text{ and }\left\vert \psi
^{\prime \prime \prime }\right\vert \leq \frac{C}{R^{3}}.  \label{c3}
\end{equation}%
Composing it with $\rho ,$ we obtain a cut-off function $\psi \left(
x\right) =\psi \left( \rho \left( x\right) \right) .$ According to Lemma \ref%
{r} it satisfies 
\begin{equation}
\left\vert \nabla \psi \right\vert +\left\vert \Delta \psi \right\vert \leq 
\frac{C}{R}\text{ \ and }\left\vert \nabla \left( \Delta \psi \right)
\right\vert \leq \frac{C}{R}\left( \left\vert \nabla ^{2}\rho \right\vert
+1\right) .  \label{c4}
\end{equation}%
Throughout this section, 
\begin{equation}
D\left( t\right) =\left\{ x\in M:\rho \left( x\right) <t\right\} .
\label{c5}
\end{equation}%
Obviously, $\psi =1$ on $D\left( R\right) $ and $\psi =0$ on $M\setminus
D\left( 2R\right) $.

We now assume that $\left( M,g\right) $ has positive spectrum, i.e., 
\begin{equation*}
\lambda _{1}\left( M\right) >0.
\end{equation*}%
In particular, $\left( M,g\right) $ is nonparabolic, that is, it admits a
positive symmetric Green's function. Then, for a given smooth connected
bounded domain $D\subset M,$ by Li-Tam \cite{LT1}, there a exists a barrier
function $f>0$ on $M\setminus D$ satisfying

\begin{eqnarray}
&&\Delta f=0\text{ \ on }M\setminus D  \label{barrier} \\
&&f=1\text{ on }\partial D  \notag \\
&&\liminf_{x\rightarrow \infty }f\left( x\right) =0.  \notag
\end{eqnarray}%
Moreover, such $f$ can be obtained as a limit $f=\lim_{R_{i}\rightarrow
\infty }f_{i},$ where each $f_{i}$ is a harmonic function on $%
D(R_{i})\setminus D$ with $f_{i}=1$ on $\partial D$ and $f_{i}=0$ on $%
\partial D\left( R_{i}\right) .$ It is easy to see that $\int_{M\setminus
D}\,\left\vert \nabla f\right\vert ^{2}<\infty .$ In particular, by \cite{LW}%
, for all $0<t\leq 1$ and $r\geq R_{0}$ with $D\subset D(R_{0}),$

\begin{equation}
-\int_{\partial D(r)}\,\frac{\partial f}{\partial \nu }=\int_{\{f=t\}}\,%
\left\vert \nabla f\right\vert =C_{0}>0,  \label{c6}
\end{equation}%
where $\nu $ is the outward unit normal vector of $\partial D(r).$

According to Cheng-Yau gradient estimate \cite{CY0}, when the Ricci
curvature of $M$ is bounded below by a constant, 
\begin{equation}
\left\vert \nabla \ln f\right\vert \leq C\text{ \ on }M\setminus D.
\label{CY}
\end{equation}

We now consider the function $u$ given by 
\begin{equation}
u=f\psi .  \label{c7}
\end{equation}%
It is straightforward to check the following lemma by (\ref{CY}) and (\ref%
{c4}).

\begin{lemma}
\label{u}Let $\left( M,g\right) $ be an $n$-dimensional complete noncompact
Riemannian manifold with Ricci curvature bounded below by a constant. Then $%
u $ is harmonic on $D\left( R\right) \setminus D$ and%
\begin{eqnarray*}
\left\vert \nabla u\right\vert &\leq &C\left( u+\frac{1}{R}\right) \\
\left\vert \Delta u\right\vert &\leq &\frac{C}{R} \\
\left\vert \nabla \left( \Delta u\right) \right\vert &\leq &\frac{C}{R}%
\left( \left\vert \nabla ^{2}\rho \right\vert ^{2}+\left\vert \nabla
^{2}f\right\vert ^{2}+1\right)
\end{eqnarray*}%
on $M\setminus D.$
\end{lemma}

We extend both $f$ and $u$ to $M$ by setting $f=1$ and $u=1$ in $D.$ For $%
0<t<1,$ denote by $L\left( t,\infty \right) $ the connected component of the
super level set $\left\{ u>t\right\} $ that contains $D.$ We furthermore
denote with 
\begin{equation*}
L\left( a,b\right) =L\left( a,\infty \right) \cap \left\{ a<u<b\right\} .
\end{equation*}

Note that all bounded components of $M\setminus D$ are contained in $L\left(
t,\infty \right) $ as well. This is because each of them shares the boundary
with $D$ and $f=1$ there. Let

\begin{equation}
l\left( t\right) =\partial L\left( t,\infty \right).  \label{c8}
\end{equation}%
Since $u$ has compact support in $M,$ the closure $\overline{L\left(
t,\infty \right) }$ of $L\left( t,\infty \right) $ and its boundary $l\left(
t\right) $ are compact in $M$.

\begin{lemma}
\label{lb}Let $\left( M,g\right) $ be an $n$-dimensional complete Riemannian
manifold, and let $M_{t}$ be the union of all unbounded connected components
of $M\setminus \overline{L\left( t,\infty \right) }$. Then 
\begin{equation*}
l\left( t\right) \cap \left( M\setminus \overline{M_{t}}\right) \subset
M\setminus D\left( R\right) .
\end{equation*}
\end{lemma}

\begin{proof}
Write 
\begin{equation*}
M\setminus \overline{L\left( t,\infty \right) }=M_{t}\cup \Omega _{t},
\end{equation*}%
where $\Omega _{t}$ is the union of all bounded components of $M\setminus 
\overline{L\left( t,\infty \right) }.$ Clearly, 
\begin{equation*}
l\left( t\right) \cap \left( M\setminus \overline{M_{t}}\right) \subset
\partial \Omega _{t}\subset l\left( t\right) .
\end{equation*}

Note that $\partial M_{t}\subset l\left( t\right) $ and $u\leq f.$ It
follows that $f\geq t$ on $\partial M_{t}=\partial \left( M\setminus 
\overline{M_{t}}\right) .$ Since the function $f$ is superharmonic on the
bounded set $M\setminus \overline{M_{t}},$ the strong maximum principle
shows that $f>t$ in $M\setminus \overline{M_{t}}.$ In particular, we have $%
f>t$ on $\overline{\Omega }_{t}\setminus \overline{M_{t}}.$

Now for $x_{1}\in \left( \partial \Omega _{t}\cap D\left( R\right) \right)
\setminus \overline{M_{t}},$ we would have $u(x_{1})>t$, as $f>t$ on $%
\overline{\Omega }_{t}\setminus \overline{M_{t}}$ and $u=f$ on $D\left(
R\right) .$ This contradicts with $\partial \Omega _{t}\subset l\left(
t\right) .$

Therefore, $\left( \partial \Omega _{t}\cap D\left( R\right) \right)
\setminus \overline{M_{t}}=\varnothing .$ In conclusion, we have%
\begin{equation*}
l\left( t\right) \cap \left( M\setminus \overline{M_{t}}\right) \subset
M\setminus D\left( R\right) .
\end{equation*}
\end{proof}

We need the following results concerning the connected components of $%
l\left( t\right) \cap \overline{M_{t}}$ under the topological assumption (A)
or (B) in Theorem \ref{A1}.

\begin{lemma}
\label{ls}Let $\left( M,g\right) $ be an $n$-dimensional manifold such that
its homology $H_{n-1}\left( M,\mathbb{Z}\right) $ contains no spherical
classes. Let $M_{t} $ be the union of all unbounded connected components of $%
M\setminus \overline{L\left( t,\infty \right) }.$ Then either no component
of $l\left( t\right) \cap \overline{M_{t}}$ is homeomorphic to $\mathbb{S}%
^{n-1},$ or 
\begin{equation*}
l\left( t\right) \cap \overline{M_{t}}=\partial M_{t}\text{ \ is connected}.
\end{equation*}
\end{lemma}

\begin{proof}
Note that 
\begin{equation*}
M\setminus l\left( t\right) =L\left( t,\infty \right) \cup M_{t}\cup \Omega
_{t},
\end{equation*}%
where $\Omega _{t}$ is the union of all bounded connected components of $%
M\setminus \overline{L\left( t,\infty \right) }$. Assume that $l_{1}\left(
t\right) \subset l\left( t\right) \cap \overline{M_{t}}$ is spherical. As $%
H_{n-1}\left( M,\mathbb{Z}\right) $ contains no spherical classes, it
follows that $M\setminus l_{1}\left( t\right) $ has a bounded connected
component. However, since $M_{t}$ is unbounded, and as $l_{1}\left( t\right)
\subset \overline{M_{t}}$, we conclude that this bounded component must be $%
L\left( t,\infty \right) \cup \overline{\Omega _{t}}$. In particular, 
\begin{equation*}
l\left( t\right) \cap \overline{M_{t}}=l_{1}\left( t\right)
\end{equation*}%
is connected.
\end{proof}

In the case that $M$ has finite first Betti number and finitely many, say $%
m, $ ends, the smooth connected bounded domain $D$ may be chosen large
enough to contain all representatives of the first homology group $%
H_{1}\left( M\right).$ Moreover, $M\setminus D$ has exactly $m$ unbounded
connected components. Then the same holds true for $M\setminus \tilde{D}$
for any bounded domain $\tilde{D}$ with $D\subset \tilde{D}.$ We have the
following estimate on the number of components of $l\left( t\right) $, cf. 
\cite{MW}.

\begin{lemma}
\label{l} Let $\left( M,g\right) $ be an $n$-dimensional complete Riemannian
manifold with $m$ ends and finite first Betti number $b_{1}(M).$ Then for
any $0<t<1,$%
\begin{equation*}
l\left( t\right) \cap \overline{M_{t}}=\partial M_{t}\text{ has }m\text{
connected components,}
\end{equation*}%
where $M_{t}$ is the union of all unbounded connected components of $%
M\setminus \overline{L\left( t,\infty \right) }$.
\end{lemma}

\begin{proof}
Recall that all representatives of the first homology $H_{1}\left( M\right) $
lie in $D$ and $M\setminus D$ has exactly $m$ unbounded components.
Therefore, as $D\subset L\left( t,\infty \right) $, then $L\left( t,\infty
\right) $ contains all representatives of $H_{1}\left( M\right) $ and $M_{t}$
has $m$ connected components for any $0<t<1.$ Note also $\partial
M_{t}\subset l\left( t\right) .$

For fixed $0<\delta <1$ such that $t+\delta <1,$ define $U$ to be the union
of $L\left( t,\infty \right) $ with all bounded components of $M\setminus 
\overline{L\left( t+\delta ,\infty \right) }$ and $V$ the union of all
unbounded components of $M\setminus \overline{L\left( t+\delta ,\infty
\right) }.$ Since $M=U\cup V,$ we have the following Mayer-Vietoris sequence%
\begin{equation*}
H_{1}\left( U\right) \oplus H_{1}\left( V\right) \overset{j_{\ast }}{%
\rightarrow }H_{1}\left( M\right) \overset{\partial }{\rightarrow }%
H_{0}\left( U\cap V\right) \overset{i_{\ast }}{\rightarrow }H_{0}\left(
U\right) \oplus H_{0}\left( V\right) \overset{j_{\ast }^{\prime }}{%
\rightarrow }H_{0}(M).
\end{equation*}%
The map $j_{\ast }$ is onto because all representatives of $H_{1}\left(
M\right) $ lie inside $U$. The map $j_{\ast }^{\prime }$ is also onto. Note
also that $V$ has $m$ components and $U$ is connected. The latter is true
because each component of $M\setminus \overline{L\left( t+\delta ,\infty
\right) }$ intersects with $L\left( t,\infty \right) $ as $\overline{L\left(
t+\delta ,\infty \right) }\subset L\left( t,\infty \right) .$ We therefore
obtain the short exact sequence 
\begin{equation*}
0\rightarrow H_{0}\left( U\cap V\right) \rightarrow \mathbb{Z}\oplus
..\oplus \mathbb{Z\rightarrow Z\rightarrow }0
\end{equation*}%
with $m+1$ summands. In conclusion,%
\begin{equation*}
H_{0}\left( U\cap V\right) =\mathbb{Z}\oplus ..\oplus \mathbb{Z}
\end{equation*}%
with $m$ summands. Since $\delta >0$ can be arbitrarily small, this proves
that%
\begin{equation*}
l\left( t\right) \cap \overline{M_{t}}\text{ has }m\text{ components}
\end{equation*}%
for $0<t<1.$
\end{proof}

Recall that from \cite{LW1},%
\begin{equation*}
\int_{M\setminus B_{p}\left( r\right) }f^{2}\left( x\right) dx\leq C\,e^{-2%
\sqrt{\lambda _{1}\left( M\right) }r}
\end{equation*}%
for all $r$ with $D\subset B_{p}(r).$ Therefore, by (\ref{c1}) and (\ref{c5}%
),%
\begin{equation}
\int_{M\setminus D\left( r\right) }f^{2}\left( x\right) dx\leq C\,e^{-\frac{1%
}{C}r}  \label{c14}
\end{equation}%
for all $r>R_{0}$. In particular, since $f\geq u,$ we get%
\begin{equation}
\mathrm{Vol}\left( L\left( \varepsilon ,1\right) \right) \leq \frac{1}{%
\varepsilon ^{2}}\int_{L\left( \varepsilon ,1\right) }u^{2}\leq \frac{C}{%
\varepsilon ^{2}}.  \label{c15}
\end{equation}%
Similarly,%
\begin{eqnarray}
\mathrm{Vol}\left( \left( M\setminus D\left( R\right) \right) \cap L\left(
\varepsilon ,1\right) \right) &\leq &\frac{1}{\varepsilon ^{2}}\int_{\left(
M\setminus D\left( R\right) \right) \cap L\left( \varepsilon ,1\right)
}u^{2}\left( x\right) dx  \label{c16} \\
&\leq &\frac{C}{\varepsilon ^{2}}e^{-\frac{1}{C}R}.  \notag
\end{eqnarray}%
Let us point out that $\left\{ u>t\right\} $ may have other connected
components in addition to $L\left( t,\infty \right) $. Denote with

\begin{equation}
\widetilde{L}\left( t,\infty \right) =\left( \left\{ u>t\right\} \setminus
L\left( t,\infty \right) \right) \cap L\left( \varepsilon ,1\right)
\label{c9}
\end{equation}%
the union of all such connected components contained in $L\left( \varepsilon
,1\right) $, and its boundary

\begin{equation}
\widetilde{l}\left( t\right) =\partial \widetilde{L}\left( t,\infty \right) .
\label{c10}
\end{equation}%
The following result provides an area estimate for $\widetilde{l}\left(
t\right) .$

\begin{lemma}
\label{area}Let $\left( M,g\right) $ be an $n$-dimensional complete
Riemannian manifold with Ricci curvature bounded from below by a constant.
Then for any $0<\varepsilon <1,$%
\begin{equation*}
\int_{\varepsilon }^{1}\mathrm{Area}\left( \widetilde{l}\left( t\right)
\right) dt\leq \frac{C}{\sqrt{R\varepsilon ^{3}}},
\end{equation*}%
where $C$ is a constant independent of $R$ and $\varepsilon $.
\end{lemma}

\begin{proof}
Applying the divergence theorem, and using Lemma \ref{u}, for any regular
set $\widetilde{l}\left( t\right) $ we have%
\begin{equation}
\int_{\widetilde{l}\left( t\right) }\left\vert \nabla u\right\vert =-\int_{%
\widetilde{L}\left( t,\infty \right) }\Delta u\leq \frac{C}{R}\mathrm{Vol}%
\left( L\left( \varepsilon ,\infty \right) \right) \leq \frac{C}{%
R\varepsilon ^{2}}.  \label{c11}
\end{equation}%
Since%
\begin{equation*}
1\leq \frac{1}{2}\sqrt{R\varepsilon ^{3}}\frac{\left\vert \nabla
u\right\vert }{u^{2}}+\frac{1}{2\sqrt{R\varepsilon ^{3}}}\frac{u^{2}}{%
\left\vert \nabla u\right\vert },
\end{equation*}%
it follows that%
\begin{equation}
\mathrm{Area}\left( \widetilde{l}\left( t\right) \right) \leq \frac{1}{2}%
\sqrt{R\varepsilon ^{3}}\int_{\widetilde{l}\left( t\right) }\frac{\left\vert
\nabla u\right\vert }{u^{2}}+\frac{1}{2\sqrt{R\varepsilon ^{3}}}\int_{%
\widetilde{l}\left( t\right) }\frac{u^{2}}{\left\vert \nabla u\right\vert }.
\label{c12}
\end{equation}%
Integrating (\ref{c11}) from $t=\varepsilon $ to $t=1,$ we obtain 
\begin{equation*}
\int_{\varepsilon }^{1}\int_{\widetilde{l}\left( t\right) }\frac{\left\vert
\nabla u\right\vert }{u^{2}}dt\leq \frac{C}{R\varepsilon ^{3}}.
\end{equation*}%
Using the co-area formula we estimate%
\begin{equation*}
\int_{\varepsilon }^{1}\int_{\widetilde{l}\left( t\right) }\frac{u^{2}}{%
\left\vert \nabla u\right\vert }dt\leq \int_{\varepsilon }^{1}\int_{\left\{
u=t\right\} }\frac{u^{2}}{\left\vert \nabla u\right\vert }dt\leq
\int_{M}u^{2}\leq C.
\end{equation*}%
Now integrating (\ref{c12}) in $t$ we arrive at%
\begin{equation*}
\int_{\varepsilon }^{1}\mathrm{Area}\left( \widetilde{l}\left( t\right)
\right) dt\leq \frac{C}{\sqrt{R\varepsilon ^{3}}},
\end{equation*}%
which proves the claim.
\end{proof}

We also need the following result.

\begin{lemma}
\label{I}Let $\left( M,g\right) $ be an $n$-dimensional complete Riemannian
manifold with positive spectrum $\lambda _{1}\left( M\right) >0$, and with
Ricci curvature bounded from below. For any $\varepsilon >0$ and any $%
a,b>\varepsilon $, we have 
\begin{equation*}
\left\vert \int_{L\left( a,b\right) }\left\vert \nabla u\right\vert
^{2}\varphi \left( u\right) -C_{0}\int_{a}^{b}\varphi \left( t\right)
dt\right\vert \leq \frac{C}{\sqrt{R}\varepsilon ^{2}}\sup_{\left( a,b\right)
}\left\vert \varphi \right\vert ,
\end{equation*}%
for any \thinspace integrable function $\varphi $ on $\left( a,b\right) $
with the constant $C$ independent of $R$ and $\varepsilon ,$ where $C_{0}$
is the constant in (\ref{c6}).
\end{lemma}

\begin{proof}
For $t<1$ we have by (\ref{c6}) that 
\begin{eqnarray*}
\int_{L\left( t,1\right) }\Delta u &=&-\int_{\partial D}\frac{\partial u}{%
\partial \nu }-\int_{l\left( t\right) }\left\vert \nabla u\right\vert \\
&=&C_{0}-\int_{l\left( t\right) }\left\vert \nabla u\right\vert .
\end{eqnarray*}%
However, by (\ref{c15}) 
\begin{equation*}
\left\vert \int_{L\left( t,1\right) }\Delta u\right\vert \leq \frac{C}{R}%
\mathrm{Vol}\left( L\left( \varepsilon ,1\right) \right) \leq \frac{C}{%
R\varepsilon ^{2}}.
\end{equation*}%
So we conclude 
\begin{equation}
\left\vert \int_{l\left( t\right) }\left\vert \nabla u\right\vert
-C_{0}\right\vert \leq \frac{C}{R\varepsilon ^{2}}\text{ \ for all }%
t>\varepsilon .  \label{c17}
\end{equation}%
According to the co-area formula,%
\begin{eqnarray*}
\int_{L\left( a,b\right) }\left\vert \nabla u\right\vert ^{2}\varphi \left(
u\right) &=&\int_{a}^{b}\varphi \left( t\right) \left( \int_{l\left(
t\right) }\left\vert \nabla u\right\vert \right) dt \\
&&+\int_{a}^{b}\varphi \left( t\right) \left( \int_{\widetilde{l}\left(
t\right) \cap L\left( a,b\right) }\left\vert \nabla u\right\vert \right) dt.
\end{eqnarray*}%
Finally, Lemma \ref{u} and Lemma \ref{area} imply that 
\begin{equation*}
\int_{a}^{b}\left\vert \varphi \left( t\right) \right\vert \left( \int_{%
\widetilde{l}\left( t\right) \cap L\left( a,b\right) }\left\vert \nabla
u\right\vert \right) dt\leq \frac{C}{\sqrt{R\varepsilon ^{3}}}\sup_{\left(
a,b\right) }\left\vert \varphi \right\vert .
\end{equation*}%
The desired result then follows immediately.
\end{proof}

From now on, we restrict to the case that $M$ is of dimension $3.$ First, we
have the following corollary of Lemma \ref{ls} and Lemma \ref{l}.

\begin{corollary}
\label{Int}Let $\left( M,g\right) $ be a three-dimensional complete
Riemannian manifold. Then on any regular level $l\left( t\right) $ as
defined in (\ref{c8}), 
\begin{equation*}
\int_{l\left( t\right) \cap \overline{M_{t}}}S_{t}\leq C
\end{equation*}%
for a constant $C$ independent of $t\in \left( 0,1\right),$ where $S_{t}$
denotes the scalar curvature of $l\left( t\right),$ provided that either (A)
or (B) holds.\newline
(A) $H_{2}\left( M,\mathbb{Z}\right) $ contains no spherical classes.\newline
(B) $M$ has finitely many ends and finite first Betti number $b_{1}(M).$%
\newline
\end{corollary}

\begin{proof}
According to the Gauss-Bonnet theorem, on each regular connected component $%
l_{k}\left( t\right) $ of $l\left( t\right) ,$%
\begin{equation*}
\int_{l_{k}\left( t\right) }S_{t}=4\pi \chi \left( l_{k}\left( t\right)
\right) .
\end{equation*}%
If $H_{2}\left( M,\mathbb{Z}\right) $ contains no spherical classes, then
Lemma \ref{ls} implies that 
\begin{equation*}
\int_{l\left( t\right) \cap \overline{M_{t}}}S_{t}\leq 8\pi .
\end{equation*}%
If $M$ has $m$ ends and finite first Betti number, then Lemma \ref{l} says
that 
\begin{equation*}
\int_{l\left( t\right) \cap \overline{M_{t}}}S_{t}\leq 8\pi m\leq C.
\end{equation*}%
In either case, the constant is independent of $t\in \left( 0,1\right) .$
\end{proof}

The following result relies on an idea from \cite{SY} and follows as in \cite%
{J} or Lemma 4.1 in \cite{BKKS}. For the sake of completeness, we include
details here.

\begin{lemma}
\label{RicS}Let $\left( M,g\right) $ be a three-dimensional complete
Riemannian manifold with scalar curvature $S$ and $u$ a smooth function on $%
M.$ Then on each regular level set $\{u=t\}$ of $u,$%
\begin{eqnarray*}
\mathrm{Ric}\left( \nabla u,\nabla u\right) \left\vert \nabla u\right\vert
^{-2} &=&\frac{1}{2}S-\frac{1}{2}S_{t}+\frac{1}{2}\frac{1}{\left\vert \nabla
u\right\vert ^{2}}\left( \left\vert \nabla \left\vert \nabla u\right\vert
\right\vert ^{2}-\left\vert \nabla ^{2}u\right\vert ^{2}\right) \\
&&+\frac{1}{2}\frac{1}{\left\vert \nabla u\right\vert ^{2}}\left( \left(
\Delta u\right) ^{2}-2\frac{\left\langle \nabla \left\vert \nabla
u\right\vert ,\nabla u\right\rangle }{\left\vert \nabla u\right\vert }\Delta
u+\left\vert \nabla \left\vert \nabla u\right\vert \right\vert ^{2}\right) ,
\end{eqnarray*}%
where $S_{t}$ denotes the scalar curvature of the surface $\{u=t\}.$
\end{lemma}

\begin{proof}
On a regular level set $\{u=t\}$ of $u,$ its unit normal vector is given by%
\begin{equation*}
e_{1}=\frac{\nabla u}{\left\vert \nabla u\right\vert }.
\end{equation*}%
Choose unit vectors $\left\{ e_{2},e_{3}\right\} $ tangent to $\{u=t\}$ such
that $\left\{ e_{1},e_{2},e_{3}\right\} $ forms a local orthonormal frame on 
$M.$ The second fundamental form and the mean curvature of $\{u=t\}$ are
then given by%
\begin{equation*}
h_{ab}=\frac{u_{ab}}{\left\vert \nabla u\right\vert }\text{ \ and }H=\frac{%
\Delta u-u_{11}}{\left\vert \nabla u\right\vert },\text{ \ respectively, }
\end{equation*}%
where indices $a$ and $b$ range from $2$ to $3.$ By the Gauss curvature
equation, we have%
\begin{equation*}
S_{t}=S-2R_{11}+H^{2}-\left\vert h\right\vert ^{2}.
\end{equation*}%
Therefore,%
\begin{eqnarray*}
&&2\mathrm{Ric}\left( \nabla u,\nabla u\right) \left\vert \nabla
u\right\vert ^{-2} \\
&=&S-S_{t}+\frac{1}{\left\vert \nabla u\right\vert ^{2}}\left( \left( \Delta
u-u_{11}\right) ^{2}-\left\vert u_{ab}\right\vert ^{2}\right) \\
&=&S-S_{t}+\frac{1}{\left\vert \nabla u\right\vert ^{2}}\left( \left( \Delta
u\right) ^{2}-2u_{11}\left( \Delta u\right) +2\left\vert \nabla \left\vert
\nabla u\right\vert \right\vert ^{2}-\left\vert u_{ij}\right\vert
^{2}\right) .
\end{eqnarray*}%
This proves the result.
\end{proof}

Finally, we recall the following well known Kato inequality for harmonic
functions (cf. \cite{MW}). 
\begin{equation}
\left\vert \nabla ^{2}f\right\vert ^{2}\geq \frac{3}{2}\left\vert \nabla
\left\vert \nabla f\right\vert \right\vert ^{2}\text{ \ on }M\setminus D.
\label{Kato}
\end{equation}

We are now ready to prove Theorem \ref{A1} which is restated below.

\begin{theorem}
Let $\left( M,g\right) $ be a three-dimensional complete Riemannian manifold
with scalar curvature $S\geq -6K$ on $M$ for some nonnegative constant $K.$
Assume that the Ricci curvature of $M$ is bounded from below by a constant.
Then the bottom spectrum of $M$ satisfies%
\begin{equation*}
\lambda _{1}\left( M\right) \leq K
\end{equation*}
provided that either (A) or (B) holds.\newline
(A) $H_{2}\left( M,\mathbb{Z}\right) $ contains no spherical classes.\newline
(B) $M$ has finitely many ends and finite first Betti number.
\end{theorem}

\begin{proof}
For given small $0<\varepsilon <1,$ let $R>R_{0}$ and%
\begin{equation}
R>\frac{1}{\varepsilon ^{6}},  \label{c13}
\end{equation}%
where $R_{0}$ is chosen to be large enough so that $D\subset D(R_{0})$. In
the following, we use $C$ to denote a positive constant that is independent
of $R$ and $\varepsilon ,$ while its value may change from line to line. To
prove the theorem, we may assume that $\lambda _{1}\left( M\right) >0.$ In
particular, $\left( M,g\right) $ is nonparabolic.

Let $f$ be the Li-Tam barrier function defined by (\ref{barrier}) and set $%
u=f\psi $ as in (\ref{c7}). Let 
\begin{equation*}
\phi (x)=\left\{ 
\begin{array}{c}
\phi \left( u\left( x\right) \right) \\ 
0%
\end{array}%
\right. 
\begin{array}{c}
\text{on }L\left( \varepsilon ,\infty \right) \\ 
\text{on }M\setminus L\left( \varepsilon ,\infty \right) .%
\end{array}%
\end{equation*}%
Here the function $\phi \left( t\right) $ is smooth, $\phi \left( t\right)
=1 $ for $2\varepsilon \leq t\leq 1$ and $\phi \left( t\right) =0$ for $%
t<\varepsilon .$ Moreover,%
\begin{equation}
\left\vert \phi ^{\prime }\right\vert \left( t\right) \leq \frac{C}{%
\varepsilon }\text{ \ and }\left\vert \phi ^{\prime \prime }\right\vert
\left( t\right) \leq \frac{C}{\varepsilon ^{2}}\text{ \ for }\varepsilon
<t<2\varepsilon .  \label{phi}
\end{equation}

Clearly, by definition, the function $\phi \left( x\right) $ satisfies $\phi
=1$ on $L\left( 2\varepsilon ,\infty \right) $ and $\phi =0$ on $M\setminus
L\left( \varepsilon ,\infty \right) .$

We first prove the following inequality. 
\begin{equation*}
\text{\textbf{Claim 1: \ \ }}\int_{M\setminus D}\left\vert \nabla
u\right\vert ^{3}u^{-2}\phi ^{2}\leq 8K^{\frac{3}{2}}\int_{M\setminus
D}u\phi ^{2}+C.
\end{equation*}

Applying the Bochner formula and the inequality $\left\langle \nabla \Delta
u,\nabla u\right\rangle \geq -\left\vert \nabla \left( \Delta u\right)
\right\vert \left\vert \nabla u\right\vert $, we have%
\begin{equation*}
\Delta \left\vert \nabla u\right\vert \geq \left( \left\vert
u_{ij}\right\vert ^{2}-\left\vert \nabla \left\vert \nabla u\right\vert
\right\vert ^{2}\right) \left\vert \nabla u\right\vert ^{-1}+\mathrm{Ric}%
\left( \nabla u,\nabla u\right) \left\vert \nabla u\right\vert
^{-1}-\left\vert \nabla \left( \Delta u\right) \right\vert
\end{equation*}%
on $M\setminus D$ and whenever $\left\vert \nabla u\right\vert \neq 0.$
Hence, it follows that%
\begin{gather*}
\int_{M\setminus D}\left( \Delta \left\vert \nabla u\right\vert \right) \phi
^{2}\geq \int_{M\setminus D}\left( \left\vert u_{ij}\right\vert
^{2}-\left\vert \nabla \left\vert \nabla u\right\vert \right\vert
^{2}\right) \left\vert \nabla u\right\vert ^{-1}\phi ^{2} \\
+\int_{M\setminus D}\mathrm{Ric}\left( \nabla u,\nabla u\right) \left\vert
\nabla u\right\vert ^{-1}\phi ^{2}-\int_{M\setminus D}\left\vert \nabla
\left( \Delta u\right) \right\vert \phi ^{2}.
\end{gather*}%
According to Lemma \ref{u},%
\begin{equation}
\int_{M\setminus D}\left\vert \nabla \left( \Delta u\right) \right\vert \phi
^{2}\leq \frac{C}{R}\int_{M\setminus D}\left( \left\vert \nabla ^{2}\rho
\right\vert ^{2}+\left\vert \nabla ^{2}f\right\vert ^{2}+1\right) \phi ^{2}.
\label{c20}
\end{equation}%
Integrating by parts gives%
\begin{eqnarray*}
\int_{M\setminus D}\left\vert \nabla ^{2}\rho \right\vert ^{2}\phi ^{2}
&=&-\int_{M\setminus D}\left\langle \nabla \left( \Delta \rho \right)
,\nabla \rho \right\rangle \phi ^{2}-\int_{M\setminus D}\mathrm{Ric}\left(
\nabla \rho ,\nabla \rho \right) \phi ^{2} \\
&&-2\int_{M\setminus D}\rho _{ij}\rho _{i}\phi _{j}\phi -\int_{\partial
D}\rho _{ij}\rho _{i}\nu _{j}\phi ^{2}.
\end{eqnarray*}%
By Lemma \ref{r} and (\ref{c15}) \ we get 
\begin{eqnarray*}
-\int_{M\setminus D}\left\langle \nabla \left( \Delta \rho \right) ,\nabla
\rho \right\rangle \phi ^{2} &\leq &C\int_{M\setminus D}\left\vert \nabla
^{2}\rho \right\vert \phi ^{2} \\
&\leq &\frac{1}{4}\int_{M\setminus D}\left\vert \nabla ^{2}\rho \right\vert
^{2}\phi ^{2}+C\,\int_{M\setminus D}\phi ^{2} \\
&\leq &\frac{1}{4}\int_{M\setminus D}\left\vert \nabla ^{2}\rho \right\vert
^{2}\phi ^{2}+\frac{C}{\varepsilon ^{2}}.
\end{eqnarray*}%
Since Ricci curvature is bounded from below by a constant, we similarly have%
\begin{equation*}
-\int_{M\setminus D}\mathrm{Ric}\left( \nabla \rho ,\nabla \rho \right) \phi
^{2}\leq C\int_{M\setminus D}\left\vert \nabla \rho \right\vert ^{2}\phi
^{2}\leq \frac{C}{\varepsilon ^{2}}.
\end{equation*}%
Finally, we have 
\begin{eqnarray*}
-2\int_{M\setminus D}\rho _{ij}\rho _{i}\phi _{j}\phi &\leq &\frac{1}{4}%
\int_{M\setminus D}\left\vert \nabla ^{2}\rho \right\vert ^{2}\phi
^{2}+4\int_{L\left( \varepsilon ,1\right) }\left\vert \nabla \phi
\right\vert ^{2} \\
&\leq &\frac{1}{4}\int_{M\setminus D}\left\vert \nabla ^{2}\rho \right\vert
^{2}\phi ^{2}+\frac{C}{\varepsilon },
\end{eqnarray*}%
where in the last line we have used (\ref{phi}) and Lemma \ref{I}. In
conclusion, this proves that 
\begin{equation*}
\int_{M\setminus D}\left\vert \nabla ^{2}\rho \right\vert ^{2}\phi ^{2}\leq 
\frac{C}{\varepsilon ^{2}}.
\end{equation*}

Similarly, we see that%
\begin{equation*}
\int_{M\setminus D}\left\vert \nabla ^{2}f\right\vert ^{2}\phi ^{2}\leq C.
\end{equation*}%
We conclude from above and (\ref{c20}) that%
\begin{equation}
\int_{M\setminus D}\left\vert \nabla \left( \Delta u\right) \right\vert \phi
^{2}\leq \frac{C}{R\varepsilon ^{2}}.  \label{c21}
\end{equation}%
Thus, we have proved that%
\begin{eqnarray}
&&\int_{M\setminus D}\left( \Delta \left\vert \nabla u\right\vert \right)
\phi ^{2}  \label{c22} \\
&\geq &\int_{M\setminus D}\left( \left\vert u_{ij}\right\vert
^{2}-\left\vert \nabla \left\vert \nabla u\right\vert \right\vert ^{2}+%
\mathrm{Ric}\left( \nabla u,\nabla u\right) \right) \left\vert \nabla
u\right\vert ^{-1}\phi ^{2}-C\varepsilon .  \notag
\end{eqnarray}

Using the co-area formula and noting that $\phi =0$ outside $L(\varepsilon
,\infty ),$ we have%
\begin{eqnarray}
&&\int_{M\setminus D}\left( \left\vert u_{ij}\right\vert ^{2}-\left\vert
\nabla \left\vert \nabla u\right\vert \right\vert ^{2}+\mathrm{Ric}\left(
\nabla u,\nabla u\right) \right) \left\vert \nabla u\right\vert ^{-1}\phi
^{2}  \label{c23} \\
&=&\int_{\varepsilon }^{1}\phi ^{2}\left( t\right) \int_{\left\{ u=t\right\}
\cap L\left( \varepsilon ,\infty \right) }\left( \left\vert
u_{ij}\right\vert ^{2}-\left\vert \nabla \left\vert \nabla u\right\vert
\right\vert ^{2}+\mathrm{Ric}\left( \nabla u,\nabla u\right) \right)
\left\vert \nabla u\right\vert ^{-2}dt  \notag \\
&=&\int_{\varepsilon }^{1}\phi ^{2}\left( t\right) \int_{l\left( t\right)
}\left( \left\vert u_{ij}\right\vert ^{2}-\left\vert \nabla \left\vert
\nabla u\right\vert \right\vert ^{2}+\mathrm{Ric}\left( \nabla u,\nabla
u\right) \right) \left\vert \nabla u\right\vert ^{-2}dt  \notag \\
&&+\int_{\varepsilon }^{1}\phi ^{2}\left( t\right) \int_{\widetilde{l}\left(
t\right) }\left( \left\vert u_{ij}\right\vert ^{2}-\left\vert \nabla
\left\vert \nabla u\right\vert \right\vert ^{2}+\mathrm{Ric}\left( \nabla
u,\nabla u\right) \right) \left\vert \nabla u\right\vert ^{-2}dt,  \notag
\end{eqnarray}%
where $l(t)$ and $\widetilde{l}(t)$ are defined in (\ref{c8}) and (\ref{c10}%
), respectively. Since $\left\vert u_{ij}\right\vert ^{2}-\left\vert \nabla
\left\vert \nabla u\right\vert \right\vert ^{2}\geq 0$ on $\widetilde{l}%
\left( t\right) \setminus D\left( R\right) $ and $u$ is harmonic on $D\left(
R\right) $, by Lemma \ref{area} and (\ref{Kato}) the last term is estimated
as 
\begin{eqnarray*}
&&\int_{\varepsilon }^{1}\phi ^{2}\left( t\right) \int_{\widetilde{l}\left(
t\right) }\left( \left\vert u_{ij}\right\vert ^{2}-\left\vert \nabla
\left\vert \nabla u\right\vert \right\vert ^{2}+\mathrm{Ric}\left( \nabla
u,\nabla u\right) \right) \left\vert \nabla u\right\vert ^{-2}dt \\
&\geq &\frac{1}{2}\int_{\varepsilon }^{1}\phi ^{2}\left( t\right) \int_{%
\widetilde{l}\left( t\right) \cap D\left( R\right) }\left\vert \nabla
\left\vert \nabla u\right\vert \right\vert ^{2}\left\vert \nabla
u\right\vert ^{-2}dt-C\int_{\varepsilon }^{1}\mathrm{Area}\left( \widetilde{l%
}\left( t\right) \right) dt \\
&\geq &\frac{1}{2}\int_{\varepsilon }^{1}\phi ^{2}\left( t\right) \int_{%
\widetilde{l}\left( t\right) \cap D\left( R\right) }\left\vert \nabla
\left\vert \nabla u\right\vert \right\vert ^{2}\left\vert \nabla
u\right\vert ^{-2}dt-\frac{C}{\sqrt{R\varepsilon ^{3}}}.
\end{eqnarray*}%
Rewrite the second last term in (\ref{c23}) as 
\begin{eqnarray*}
&&\int_{\varepsilon }^{1}\phi ^{2}\left( t\right) \int_{l\left( t\right)
}\left( \left\vert u_{ij}\right\vert ^{2}-\left\vert \nabla \left\vert
\nabla u\right\vert \right\vert ^{2}+\mathrm{Ric}\left( \nabla u,\nabla
u\right) \right) \left\vert \nabla u\right\vert ^{-2}dt \\
&=&\int_{\varepsilon }^{1}\phi ^{2}\left( t\right) \int_{l\left( t\right)
\cap \overline{M_{t}}}\left( \left\vert u_{ij}\right\vert ^{2}-\left\vert
\nabla \left\vert \nabla u\right\vert \right\vert ^{2}+\mathrm{Ric}\left(
\nabla u,\nabla u\right) \right) \left\vert \nabla u\right\vert ^{-2}dt \\
&&+\int_{\varepsilon }^{1}\phi ^{2}\left( t\right) \int_{l\left( t\right)
\cap \left( M\setminus \overline{M_{t}}\right) }\left( \left\vert
u_{ij}\right\vert ^{2}-\left\vert \nabla \left\vert \nabla u\right\vert
\right\vert ^{2}+\mathrm{Ric}\left( \nabla u,\nabla u\right) \right)
\left\vert \nabla u\right\vert ^{-2}dt.
\end{eqnarray*}

Applying the inequality $\left\vert u_{ij}\right\vert ^{2}-\left\vert \nabla
\left\vert \nabla u\right\vert \right\vert ^{2}\geq 0,$ together with the
fact that the Ricci curvature is bounded from below, we conclude 
\begin{eqnarray*}
&&\int_{\varepsilon }^{1}\phi ^{2}\left( t\right) \int_{l\left( t\right)
\cap \left( M\setminus \overline{M_{t}}\right) }\left( \left\vert
u_{ij}\right\vert ^{2}-\left\vert \nabla \left\vert \nabla u\right\vert
\right\vert ^{2}+\mathrm{Ric}\left( \nabla u,\nabla u\right) \right)
\left\vert \nabla u\right\vert ^{-2}dt \\
&\geq &-C\int_{\varepsilon }^{1}\phi ^{2}\left( t\right) \int_{l\left(
t\right) \cap \left( M\setminus D\left( R\right) \right) }dt \\
&\geq &-C\int_{M\setminus D\left( R\right) }\left\vert \nabla u\right\vert
\phi ^{2} \\
&\geq &-C\varepsilon ,
\end{eqnarray*}%
where in the second line we have used Lemma \ref{lb} and in the last line (%
\ref{c16}). Hence, (\ref{c23}) becomes 
\begin{eqnarray}
&&\int_{M\setminus D}\left( \left\vert u_{ij}\right\vert ^{2}-\left\vert
\nabla \left\vert \nabla u\right\vert \right\vert ^{2}+\mathrm{Ric}\left(
\nabla u,\nabla u\right) \right) \left\vert \nabla u\right\vert ^{-1}\phi
^{2}  \label{c24} \\
&\geq &\int_{\varepsilon }^{1}\phi ^{2}\left( t\right) \int_{l\left(
t\right) \cap \overline{M_{t}}}\left( \left\vert u_{ij}\right\vert
^{2}-\left\vert \nabla \left\vert \nabla u\right\vert \right\vert ^{2}+%
\mathrm{Ric}\left( \nabla u,\nabla u\right) \right) \left\vert \nabla
u\right\vert ^{-2}dt  \notag \\
&&+\frac{1}{2}\int_{\varepsilon }^{1}\phi ^{2}\left( t\right) \int_{%
\widetilde{l}\left( t\right) \cap D\left( R\right) }\left\vert \nabla
\left\vert \nabla u\right\vert \right\vert ^{2}\left\vert \nabla
u\right\vert ^{-2}dt-C\varepsilon .  \notag
\end{eqnarray}%
On any regular level set $l\left( t\right) $, from Lemma \ref{RicS},
Corollary \ref{Int}, and the assumption $S\geq -6K$, we have 
\begin{eqnarray*}
&&\int_{l\left( t\right) \cap \overline{M_{t}}}\left( \left\vert
u_{ij}\right\vert ^{2}-\left\vert \nabla \left\vert \nabla u\right\vert
\right\vert ^{2}+\mathrm{Ric}\left( \nabla u,\nabla u\right) \right)
\left\vert \nabla u\right\vert ^{-2} \\
&\geq &\frac{1}{2}\int_{l\left( t\right) \cap \overline{M_{t}}}\left(
\left\vert u_{ij}\right\vert ^{2}-\left\vert \nabla \left\vert \nabla
u\right\vert \right\vert ^{2}\right) \left\vert \nabla u\right\vert ^{-2} \\
&&+\frac{1}{2}\int_{l\left( t\right) \cap \overline{M_{t}}}\left( \left(
\Delta u\right) ^{2}-2\frac{\left\langle \nabla \left\vert \nabla
u\right\vert ,\nabla u\right\rangle }{\left\vert \nabla u\right\vert }\Delta
u+\left\vert \nabla \left\vert \nabla u\right\vert \right\vert ^{2}\right)
\left\vert \nabla u\right\vert ^{-2} \\
&&-3K\mathrm{Area}\left( l\left( t\right) \cap \overline{M_{t}}\right) -C.
\end{eqnarray*}%
We note that 
\begin{equation*}
\int_{\varepsilon }^{1}\phi ^{2}\left( t\right) \mathrm{Area}\left( l\left(
t\right) \cap \overline{M_{t}}\right) dt\leq \int_{\varepsilon }^{1}\phi
^{2}\left( t\right) \mathrm{Area}\left( l\left( t\right) \right) dt\leq
\int_{M\setminus D}\left\vert \nabla u\right\vert \phi ^{2}.
\end{equation*}%
Accordingly, (\ref{c24}) becomes 
\begin{eqnarray}
&&\int_{M\setminus D}\left( \left\vert u_{ij}\right\vert ^{2}-\left\vert
\nabla \left\vert \nabla u\right\vert \right\vert ^{2}+\mathrm{Ric}\left(
\nabla u,\nabla u\right) \right) \left\vert \nabla u\right\vert ^{-1}\phi
^{2}  \label{c25} \\
&\geq &\frac{1}{2}\int_{\varepsilon }^{1}\phi ^{2}\left( t\right)
\int_{l\left( t\right) \cap \overline{M_{t}}}\left( \left\vert
u_{ij}\right\vert ^{2}-\left\vert \nabla \left\vert \nabla u\right\vert
\right\vert ^{2}\right) \left\vert \nabla u\right\vert ^{-2}dt  \notag \\
&&+\frac{1}{2}\int_{\varepsilon }^{1}\phi ^{2}\left( t\right) \int_{l\left(
t\right) \cap \overline{M_{t}}}\left( \left( \Delta u\right) ^{2}-2\frac{%
\left\langle \nabla \left\vert \nabla u\right\vert ,\nabla u\right\rangle }{%
\left\vert \nabla u\right\vert }\Delta u+\left\vert \nabla \left\vert \nabla
u\right\vert \right\vert ^{2}\right) \left\vert \nabla u\right\vert ^{-2}dt 
\notag \\
&&+\frac{1}{2}\int_{\varepsilon }^{1}\phi ^{2}\left( t\right) \int_{%
\widetilde{l}\left( t\right) \cap D\left( R\right) }\left\vert \nabla
\left\vert \nabla u\right\vert \right\vert ^{2}\left\vert \nabla
u\right\vert ^{-2}dt-3K\int_{M\setminus D}\left\vert \nabla u\right\vert
\phi ^{2}-C.  \notag
\end{eqnarray}%
Using the inequality $\left\vert u_{ij}\right\vert ^{2}-\left\vert \nabla
\left\vert \nabla u\right\vert \right\vert ^{2}\geq 0$ and Lemma \ref{lb},
the first term on the right hand side can be estimated as%
\begin{eqnarray*}
&&\frac{1}{2}\int_{\varepsilon }^{1}\phi ^{2}\left( t\right) \int_{l\left(
t\right) \cap \overline{M_{t}}}\left( \left\vert u_{ij}\right\vert
^{2}-\left\vert \nabla \left\vert \nabla u\right\vert \right\vert
^{2}\right) \left\vert \nabla u\right\vert ^{-2}dt \\
&\geq &\frac{1}{2}\int_{\varepsilon }^{1}\phi ^{2}\left( t\right)
\int_{l\left( t\right) \cap D\left( R\right) }\left( \left\vert
u_{ij}\right\vert ^{2}-\left\vert \nabla \left\vert \nabla u\right\vert
\right\vert ^{2}\right) \left\vert \nabla u\right\vert ^{-2}dt \\
&\geq &\frac{1}{4}\int_{\varepsilon }^{1}\phi ^{2}\left( t\right)
\int_{l\left( t\right) \cap D\left( R\right) }\left\vert \nabla \left\vert
\nabla u\right\vert \right\vert ^{2}\left\vert \nabla u\right\vert ^{-2}dt,
\end{eqnarray*}%
where the last line is by (\ref{Kato}) and the fact that $u=f$ on $D\left(
R\right) .$ Similarly, the second term on the right hand side of (\ref{c25})
is bounded by%
\begin{eqnarray*}
&&\frac{1}{2}\int_{\varepsilon }^{1}\phi ^{2}\left( t\right) \int_{l\left(
t\right) \cap \overline{M_{t}}}\left( \left( \Delta u\right) ^{2}-2\frac{%
\left\langle \nabla \left\vert \nabla u\right\vert ,\nabla u\right\rangle }{%
\left\vert \nabla u\right\vert }\Delta u+\left\vert \nabla \left\vert \nabla
u\right\vert \right\vert ^{2}\right) \left\vert \nabla u\right\vert ^{-2}dt
\\
&\geq &\frac{1}{2}\int_{\varepsilon }^{1}\phi ^{2}\left( t\right)
\int_{l\left( t\right) \cap D\left( R\right) }\left\vert \nabla \left\vert
\nabla u\right\vert \right\vert ^{2}\left\vert \nabla u\right\vert ^{-2}dt.
\end{eqnarray*}%
Consequently, (\ref{c25}) becomes%
\begin{eqnarray}
&&\int_{M\setminus D}\left( \left\vert u_{ij}\right\vert ^{2}-\left\vert
\nabla \left\vert \nabla u\right\vert \right\vert ^{2}+\mathrm{Ric}\left(
\nabla u,\nabla u\right) \right) \left\vert \nabla u\right\vert ^{-1}\phi
^{2}  \label{c26} \\
&\geq &\frac{3}{4}\int_{\varepsilon }^{1}\phi ^{2}\left( t\right)
\int_{l\left( t\right) \cap D\left( R\right) }\left\vert \nabla \left\vert
\nabla u\right\vert \right\vert ^{2}\left\vert \nabla u\right\vert
^{-2}dt-3K\int_{M\setminus D}\left\vert \nabla u\right\vert \phi ^{2}  \notag
\\
&&+\frac{1}{2}\int_{\varepsilon }^{1}\phi ^{2}\left( t\right) \int_{%
\widetilde{l}\left( t\right) \cap D\left( R\right) }\left\vert \nabla
\left\vert \nabla u\right\vert \right\vert ^{2}\left\vert \nabla
u\right\vert ^{-2}dt-C.  \notag
\end{eqnarray}%
On the other hand, the left side of (\ref{c22}) satisfies 
\begin{equation*}
\int_{M\setminus D}\left( \Delta \left\vert \nabla u\right\vert \right) \phi
^{2}=\int_{M\setminus D}\left\vert \nabla u\right\vert \Delta \phi
^{2}-\int_{\partial D}\phi ^{2}\left\vert \nabla u\right\vert _{\nu }.
\end{equation*}%
Since $u=f$ on $D\left( R\right) \setminus D,$ we have%
\begin{equation*}
\int_{\partial D}\phi ^{2}\left\vert \nabla u\right\vert _{\nu
}=\int_{\partial D}\phi ^{2}\left\vert \nabla f\right\vert _{\nu }=C.
\end{equation*}%
Note furthermore that $\Delta \phi ^{2}=0$ on $M\setminus L\left(
\varepsilon ,2\varepsilon \right) .$ On $L\left( \varepsilon ,2\varepsilon
\right) $, we have by Lemma \ref{u} and (\ref{phi})%
\begin{eqnarray*}
\left\vert \Delta \phi ^{2}\right\vert  &\leq &\frac{C}{\varepsilon ^{2}}%
\,\left( \left\vert \Delta u\right\vert +\left\vert \nabla u\right\vert
^{2}\right)  \\
&\leq &\frac{C}{R\varepsilon ^{2}}+\frac{C}{\varepsilon ^{2}}\left\vert
\nabla u\right\vert ^{2} \\
&\leq &C\varepsilon ^{2}+\frac{C}{\varepsilon }\left\vert \nabla
u\right\vert .
\end{eqnarray*}%
Hence, Lemma \ref{I} and (\ref{c15}) imply that 
\begin{equation*}
\int_{M\setminus D}\left\vert \nabla u\right\vert \left\vert \Delta \phi
^{2}\right\vert \leq C\varepsilon ^{2}\mathrm{Vol}\left( L\left( \varepsilon
,1\right) \right) +\frac{C}{\varepsilon }\int_{L\left( \varepsilon
,2\varepsilon \right) }\left\vert \nabla u\right\vert ^{2}\leq C.
\end{equation*}%
This proves that 
\begin{equation*}
\int_{M\setminus D}\left( \Delta \left\vert \nabla u\right\vert \right) \phi
^{2}\leq C.
\end{equation*}

We therefore conclude from (\ref{c22}) and (\ref{c26}) that%
\begin{eqnarray}
&&\frac{3}{4}\int_{\varepsilon }^{1}\phi ^{2}\left( t\right) \int_{l\left(
t\right) \cap D\left( R\right) }\left\vert \nabla \left\vert \nabla
u\right\vert \right\vert ^{2}\left\vert \nabla u\right\vert ^{-2}dt
\label{c30} \\
&&+\frac{1}{2}\int_{\varepsilon }^{1}\phi ^{2}\left( t\right) \int_{%
\widetilde{l}\left( t\right) \cap D\left( R\right) }\left\vert \nabla
\left\vert \nabla u\right\vert \right\vert ^{2}\left\vert \nabla
u\right\vert ^{-2}dt  \notag \\
&\leq &3K\int_{M\setminus D}\left\vert \nabla u\right\vert \phi ^{2}+C. 
\notag
\end{eqnarray}%
Integrating by parts yields 
\begin{eqnarray*}
\int_{M\setminus D}\left\langle \nabla \left\vert \nabla u\right\vert
,\nabla u\right\rangle u^{-1}\phi ^{2} &=&-\int_{M\setminus D}\left\vert
\nabla u\right\vert u^{-1}\left( \Delta u\right) \phi ^{2}+\int_{M\setminus
D}\left\vert \nabla u\right\vert ^{3}u^{-2}\phi ^{2} \\
&&-\int_{M\setminus D}\left\vert \nabla u\right\vert u^{-1}\left\langle
\nabla u,\nabla \phi ^{2}\right\rangle -\int_{\partial D}\left\vert \nabla
u\right\vert u^{-1}u_{\nu }\phi ^{2} \\
&\geq &\int_{M\setminus D}\left\vert \nabla u\right\vert ^{3}u^{-2}\phi
^{2}-C.
\end{eqnarray*}%
In the last line we have used Lemma \ref{u} and Lemma \ref{I} to conclude
that 
\begin{equation*}
\left\vert \int_{M\setminus D}\left\vert \nabla u\right\vert
u^{-1}\left\langle \nabla u,\nabla \phi ^{2}\right\rangle \right\vert \leq 
\frac{C}{\varepsilon }\int_{L\left( \varepsilon ,2\varepsilon \right)
}\left\vert \nabla u\right\vert ^{3}u^{-1}\leq C
\end{equation*}%
as well as 
\begin{equation*}
\left\vert -\int_{M\setminus D}\left\vert \nabla u\right\vert u^{-1}\left(
\Delta u\right) \phi ^{2}\right\vert \leq \frac{C}{R}\mathrm{Vol}\left(
L(\varepsilon ,1)\right) \leq C\varepsilon 
\end{equation*}%
by Lemma \ref{u} and (\ref{c15}). Moreover, in view of (\ref{c16}), 
\begin{equation*}
\int_{M\setminus D\left( R\right) }\left\langle \nabla \left\vert \nabla
u\right\vert ,\nabla u\right\rangle u^{-1}\phi ^{2}\leq C\int_{M\setminus
D\left( R\right) }\left\vert \nabla \left\vert \nabla u\right\vert
\right\vert ^{2}\phi ^{2}+C\varepsilon .
\end{equation*}%
Let $\sigma $ be a cut-off function with support in $M\setminus D\left( 
\frac{R}{2}\right) $ so that $\sigma =1$ on $M\setminus D\left( R\right) $
and $\left\vert \nabla \sigma \right\vert \leq \frac{C}{R}$. According to
the Bochner formula, and using the Ricci curvature lower bound, we have 
\begin{eqnarray*}
\int_{M}\left\vert \nabla \left\vert \nabla u\right\vert \right\vert
^{2}\phi ^{2}\sigma ^{2} &\leq &\frac{1}{2}\int_{M}\phi ^{2}\sigma
^{2}\Delta \left\vert \nabla u\right\vert ^{2}+C\int_{M}\left\vert \nabla
u\right\vert ^{2}\phi ^{2}\sigma ^{2} \\
&=&-2\int_{M}\left\langle \nabla \left( \phi \sigma \right) ,\nabla
\left\vert \nabla u\right\vert \right\rangle \phi \sigma \left\vert \nabla
u\right\vert +C\int_{M}\left\vert \nabla u\right\vert ^{2}\phi ^{2}\sigma
^{2} \\
&\leq &\frac{1}{2}\int_{M}\left\vert \nabla \left\vert \nabla u\right\vert
\right\vert ^{2}\phi ^{2}\sigma ^{2}+C\int_{M}\left\vert \nabla u\right\vert
^{2}\left\vert \nabla \phi \right\vert ^{2}\sigma ^{2} \\
&&+C\int_{M}\left\vert \nabla u\right\vert ^{2}\left\vert \nabla \sigma
\right\vert ^{2}\phi ^{2}+C\int_{M}\left\vert \nabla u\right\vert ^{2}\phi
^{2}\sigma ^{2}.
\end{eqnarray*}%
Hence, (\ref{c16}) implies that 
\begin{equation*}
\int_{M\setminus D\left( R\right) }\left\vert \nabla \left\vert \nabla
u\right\vert \right\vert ^{2}\phi ^{2}\leq C\varepsilon .
\end{equation*}%
In conclusion, we have proved that 
\begin{equation}
\int_{M\setminus D}\left\vert \nabla u\right\vert ^{3}u^{-2}\phi ^{2}\leq
\int_{D\left( R\right) \setminus D}\left\langle \nabla \left\vert \nabla
u\right\vert ,\nabla u\right\rangle u^{-1}\phi ^{2}+C.  \label{c31}
\end{equation}

By the co-area formula, 
\begin{eqnarray*}
\int_{D\left( R\right) \setminus D}\left\langle \nabla \left\vert \nabla
u\right\vert ,\nabla u\right\rangle u^{-1}\phi ^{2} &=&\int_{\varepsilon
}^{1}\phi ^{2}\left( t\right) \int_{l\left( t\right) \cap D\left( R\right)
}\left\langle \nabla \left\vert \nabla u\right\vert ,\nabla u\right\rangle
u^{-1}\,\left\vert \nabla u\right\vert ^{-1}\,dt \\
&&+\int_{\varepsilon }^{1}\phi ^{2}\left( t\right) \int_{\widetilde{l}\left(
t\right) \cap D\left( R\right) }\left\langle \nabla \left\vert \nabla
u\right\vert ,\nabla u\right\rangle u^{-1}\,\left\vert \nabla u\right\vert
^{-1}\,dt.
\end{eqnarray*}%
The first term on the right side is estimated by 
\begin{eqnarray*}
&&\int_{\varepsilon }^{1}\phi ^{2}\left( t\right) \int_{l\left( t\right)
\cap D\left( R\right) }\left\langle \nabla \left\vert \nabla u\right\vert
,\nabla u\right\rangle u^{-1}\,\left\vert \nabla u\right\vert ^{-1}\,dt \\
&\leq &\frac{1}{2}\int_{\varepsilon }^{1}\phi ^{2}\left( t\right)
\int_{l\left( t\right) \cap D\left( R\right) }\left\vert \nabla \left\vert
\nabla u\right\vert \right\vert ^{2}\left\vert \nabla u\right\vert ^{-2}dt \\
&&+\frac{1}{2}\int_{D\left( R\right) \setminus D}\left\vert \nabla
u\right\vert ^{3}u^{-2}\phi ^{2},
\end{eqnarray*}%
and the second term by 
\begin{eqnarray*}
&&\int_{\varepsilon }^{1}\phi ^{2}\left( t\right) \int_{\widetilde{l}\left(
t\right) \cap D\left( R\right) }\left\langle \nabla \left\vert \nabla
u\right\vert ,\nabla u\right\rangle u^{-1}\,\left\vert \nabla u\right\vert
^{-1}\,dt \\
&\leq &\frac{1}{4}\int_{\varepsilon }^{1}\phi ^{2}\left( t\right) \int_{%
\widetilde{l}\left( t\right) \cap D\left( R\right) }\left\vert \nabla
\left\vert \nabla u\right\vert \right\vert ^{2}\left\vert \nabla
u\right\vert ^{-2}dt \\
&&+\int_{\varepsilon }^{1}\phi ^{2}(t)\int_{\widetilde{l}\left( t\right) \cap
D\left( R\right) }\left\vert \nabla u\right\vert ^{2}u^{-2}dt.
\end{eqnarray*}%
However, by Lemma \ref{u} and Lemma \ref{area}, 
\begin{equation*}
\int_{\varepsilon }^{1}\phi ^{2}(t)\int_{\widetilde{l}\left( t\right) \cap
D\left( R\right) }\left\vert \nabla u\right\vert ^{2}u^{-2}dt\leq
C\int_{\varepsilon }^{1}\mathrm{Area}\left( \widetilde{l}\left( t\right)
\right) dt\leq C\varepsilon .
\end{equation*}%
Hence, by (\ref{c30}), 
\begin{eqnarray*}
\int_{D\left( R\right) \setminus D}\left\langle \nabla \left\vert \nabla
u\right\vert ,\nabla u\right\rangle u^{-1}\phi ^{2} &\leq &\frac{1}{2}%
\int_{\varepsilon }^{1}\phi ^{2}\left( t\right) \int_{l\left( t\right) \cap
D\left( R\right) }\left\vert \nabla \left\vert \nabla u\right\vert
\right\vert ^{2}\left\vert \nabla u\right\vert ^{-2}dt \\
&&+\frac{1}{4}\int_{\varepsilon }^{1}\phi ^{2}\left( t\right) \int_{%
\widetilde{l}\left( t\right) \cap D\left( R\right) }\left\vert \nabla
\left\vert \nabla u\right\vert \right\vert ^{2}\left\vert \nabla
u\right\vert ^{-2}dt \\
&&+\frac{1}{2}\int_{D\left( R\right) \setminus D}\left\vert \nabla
u\right\vert ^{3}u^{-2}\phi ^{2}+C \\
&\leq &2K\int_{M\setminus D}\left\vert \nabla u\right\vert \phi ^{2}+\frac{1%
}{2}\int_{D\left( R\right) \setminus D}\left\vert \nabla u\right\vert
^{3}u^{-2}\phi ^{2}+C.
\end{eqnarray*}%
Combining this with (\ref{c31}), we conclude that 
\begin{equation*}
\int_{M\setminus D}\left\vert \nabla u\right\vert ^{3}u^{-2}\phi ^{2}\leq
4K\int_{M\setminus D}\left\vert \nabla u\right\vert \phi ^{2}+C.
\end{equation*}%
Now together with 
\begin{equation*}
4K\int_{M\setminus D}\left\vert \nabla u\right\vert \phi ^{2}\leq \frac{1}{3}%
\int_{M\setminus D}\left\vert \nabla u\right\vert ^{3}u^{-2}\phi ^{2}+\frac{%
16}{3}K^{\frac{3}{2}}\int_{M\setminus D}u\phi ^{2},
\end{equation*}%
which follows from the elementary inequality $ab\leq \frac{a^{3}}{3}+\frac{%
2b^{\frac{3}{2}}}{3}$ with $a=\left\vert \nabla u\right\vert u^{-\frac{2}{3}}
$ and $b=4Ku^{\frac{2}{3}}$, we arrive at 
\begin{equation}
\int_{M\setminus D}\left\vert \nabla u\right\vert ^{3}u^{-2}\phi ^{2}\leq
8K^{\frac{3}{2}}\int_{M\setminus D}u\phi ^{2}+C.  \label{gu}
\end{equation}%
This verifies \textbf{Claim 1}.

We now turn to the following claim. 
\begin{equation*}
\text{\textbf{Claim 2:\ \ }}\lambda _{1}\left( M\right) \int_{M\setminus
D}u\phi ^{2}\leq \frac{1}{4}\int_{M\setminus D}\left\vert \nabla
u\right\vert ^{2}u^{-1}\phi ^{2}+C.
\end{equation*}

Let $\xi $ be a Lipschitz cut-off such that $\xi =0$\ on $D\left(
R_{0}\right) $ and $\xi =1$ on $M\setminus D\left( R_{0}+1\right) $. Since
the function $u\phi ^{2}\xi ^{2}$ has compact support in $M,$ we have%
\begin{eqnarray*}
\lambda _{1}\left( M\right) \int_{M}u\phi ^{2}\xi ^{2} &\leq
&\int_{M}\left\vert \nabla \left( u^{\frac{1}{2}}\phi \xi \right)
\right\vert ^{2} \\
&=&\frac{1}{4}\int_{M}\left\vert \nabla u\right\vert ^{2}u^{-1}\phi ^{2}\xi
^{2}+\int_{M}u\left\vert \nabla \left( \phi \xi \right) \right\vert ^{2} \\
&&+\frac{1}{2}\int_{M}\left\langle \nabla u,\nabla \left( \phi ^{2}\xi
^{2}\right) \right\rangle .
\end{eqnarray*}%
According to Lemma \ref{I} we have $\int_{M\setminus D}u\left\vert \nabla
\phi \right\vert ^{2}\leq C$. Furthermore, we perform integration by parts
and get from Lemma \ref{u} and (\ref{c16}) that 
\begin{equation*}
\int_{M}\left\langle \nabla u,\nabla \left( \phi ^{2}\xi ^{2}\right)
\right\rangle =-\int_{M}\left( \Delta u\right) \phi ^{2}\xi ^{2}\leq
C\varepsilon .
\end{equation*}%
In conclusion, we have proved that 
\begin{equation*}
\lambda _{1}\left( M\right) \int_{M\setminus D}u\phi ^{2}\xi ^{2}\leq \frac{1%
}{4}\int_{M\setminus D}\left\vert \nabla u\right\vert ^{2}u^{-1}\phi ^{2}\xi
^{2}+C.
\end{equation*}%
\textbf{Claim 2} now follows.

We now show $\lambda _{1}\left( M\right) \leq K.$ Assume by contradiction
that%
\begin{equation}
\lambda _{1}\left( M\right) >K.  \label{c32}
\end{equation}

Combining \textbf{Claim 1} and \textbf{Claim 2} we conclude that%
\begin{eqnarray*}
\lambda _{1}\left( M\right) \int_{M\setminus D}u\phi ^{2} &\leq &\frac{1}{4}%
\int_{M\setminus D}\left\vert \nabla u\right\vert ^{2}u^{-1}\phi ^{2}+C \\
&\leq &\frac{1}{4}\left( \int_{M\setminus D}\left\vert \nabla u\right\vert
^{3}u^{-2}\phi ^{2}\right) ^{\frac{2}{3}}\left( \int_{M\setminus
D}u\,\phi^2\right) ^{\frac{1}{3}}+C \\
&\leq &K\int_{M\setminus D}u\phi ^{2}+C\left( \int_{M\setminus
D}u\,\phi^2\right) ^{\frac{1}{3}}+C.
\end{eqnarray*}%
By (\ref{c32}), it implies that $\int_{M\setminus D}u\phi ^{2}\leq C.$

For arbitrary large $r$ such that $D\subset B_{p}(r),$ let%
\begin{equation*}
t\left( r\right) =\min_{x\in \partial B_{p}\left( r\right) }f\left( x\right).
\end{equation*}%
The maximum principle implies that 
\begin{equation*}
B_{p}\left( r\right) \subset \left\{ f>t\left( r\right) \right\} .
\end{equation*}%
So for $\varepsilon <\frac{1}{2}t\left( r\right) $ and all $R>\max \{R_{0},%
\frac{1}{\varepsilon ^{4}}\}$ large, 
\begin{equation*}
B_{p}\left( r\right) \setminus D\subset D\left( R\right) \cap L\left(
2\varepsilon ,1\right) .
\end{equation*}%
As $u=f$ on $D\left( R\right) \cap L\left(2\varepsilon ,1\right),$ it
follows that 
\begin{equation}
\int_{B_{p}\left( r\right) \setminus D }f\leq \int_{M\setminus D}u\phi ^{2}
\leq C  \label{c34}
\end{equation}
with the constant $C$ independent of $r.$ On the other hand, the co-area
formula and (\ref{CY}), (\ref{c6}) imply 
\begin{eqnarray*}
\int_{B_{p}\left( r\right) \setminus D }f &\geq& C\, \int_{B_{p}\left(
r\right) \setminus D}\left\vert \nabla f\right\vert \\
&\geq& C\,\int_{R_0}^r dt\int_{\partial B_p(t)}\left\vert \nabla f\right\vert
\\
&\geq& C\,C_0\,(r-R_0)
\end{eqnarray*}%
for any $r>R_0.$ This contradicts (\ref{c34}). Therefore, (\ref{c32}) does
not hold. The theorem is proved.
\end{proof}

\section{Manifolds with maximal bottom spectrum\label{Max}}

In this section we prove Theorem \ref{A2}. We start by making some general
discussions without restriction on the dimension $n$ of $M.$ Note that for a
manifold $M$ with a finite volume end and positive bottom spectrum $\lambda
_{1}(M)>0,$ since its volume must be infinite, it has an infinite volume end
as well. By Nakai \cite{N} and Li-Tam \cite{LT1} there exists a positive
harmonic function $f$ on $M$ such that $f$ is proper and goes to infinity on
a given finite volume end $F$ and it is bounded away from $F.$ Therefore,
without loss of generality, we may assume that $F=\{f\geq t_{0}\}$ for some
large positive number $t_{0}$ and $F$ is connected. We view $E=M\setminus F$
as an infinite volume end and denote $E\left( R\right) =E\cap D\left(
R\right) ,$ where $D(R)=\{\rho (x)<R\}$ as in Section \ref{Est} and $\rho $
be the function from Lemma \ref{r}. By Li-Wang \cite{LW1}, such $f$ may be
chosen to satisfy 
\begin{equation}
\int_{E\setminus E\left( R\right) }f^{2}\leq C\,e^{-\frac{1}{C}R}  \label{a2}
\end{equation}%
for some constant $C>0$ depending on $\lambda _{1}(M).$ In particular, 
\begin{equation}
\inf_{E}f=0\text{ and }\int_{E}f^{2}<\infty .  \label{a1} \\
\end{equation}%
According to \cite{LW}, for such $f,$ $\int_{\left\{ f=t\right\} }\left\vert
\nabla f\right\vert $ is independent of $t.$ So we normalize $f$ such that 
\begin{equation}
\int_{\left\{ f=t\right\} }\left\vert \nabla f\right\vert =1.  \label{a3}
\end{equation}

For given $R>0,$ define cut-off function $\psi $ on $M$ as in Section \ref%
{Est} by $\psi \left( x\right) =\psi \left( \rho \left( x\right) \right) $
with $\psi $ given in (\ref{c3}). Define the function 
\begin{equation}
u\left( x\right) =\left\{ 
\begin{array}{c}
\psi f \\ 
f%
\end{array}%
\right. 
\begin{array}{c}
\text{on} \\ 
\text{on}%
\end{array}%
\begin{array}{c}
E \\ 
F.%
\end{array}
\label{a4}
\end{equation}%
Note that $u$ is smooth on $M$ and satisfies 
\begin{equation*}
u=f\text{ \ on }F\cup E\left( R\right) \text{ and }u=0\text{\ on }E\setminus
E\left( 2R\right) .
\end{equation*}%
Moreover, the level set $\{u=t\}$ of $u$ is compact for each $t>0.$ Lemma %
\ref{u} implies that on the end $E$ 
\begin{eqnarray}
\left\vert \nabla u\right\vert &\leq &C\left( u+\frac{1}{R}\right)
\label{a5} \\
\left\vert \Delta u\right\vert &\leq &\frac{C}{R}  \notag \\
\left\vert \nabla \left( \Delta u\right) \right\vert &\leq &\frac{C}{R}%
\left( \left\vert \nabla ^{2}\rho \right\vert ^{2}+\left\vert \nabla
^{2}f\right\vert ^{2}+1\right) .  \notag
\end{eqnarray}%
On the end $F,$ for $M$ with Ricci curvature lower bound, by (\ref{CY}), 
\begin{equation}
\left\vert \nabla u\right\vert \leq Cu\text{ \ and }\Delta u=0.  \label{a6}
\end{equation}

Similar to the proof of Theorem \ref{A1}, define $L\left(t,\infty \right) $
to be the connected component of the set of $\{u>t\}$ that contains $F.$
Also, 
\begin{equation}
l\left( t\right) =\partial L\left( t,\infty \right) .  \label{a7}
\end{equation}%
Parallel to Lemma \ref{lb} and Lemma \ref{ls}, we have the following.

\begin{lemma}
\label{Level}Let $M_{t}$ be the union of all unbounded connected components
of $M\setminus \overline{L\left( t,\infty \right) },$ and $l_{0}\left(
t\right) $ the union of all connected components of $l\left( t\right) $ that
are homeomorphic to $\mathbb{S}^{n-1}.$ If the homology group $H_{n-1}(M,%
\mathbb{Z})$ contains no spherical classes, then 
\begin{equation*}
l_{0}\left( t\right) \subset l\left( t\right) \cap \left( M\setminus 
\overline{M_{t}}\right) \subset E\setminus D\left( R\right) .
\end{equation*}
\end{lemma}

\begin{proof}
For the first inclusion, note that as $H_{n-1}\left( M,\mathbb{Z}\right) $
contains no spherical classes, any component of $l_{0}\left( t\right) $,
being spherical, must bound a compact component of $M\setminus \overline{%
L\left( t,\infty \right) }.$ Obviously, such a component is disjoint from $%
M_{t}.$

For the second inclusion, the same argument as in Lemma \ref{lb} implies that%
\begin{equation*}
l\left( t\right) \cap \left( M\setminus \overline{M_{t}}\right) \subset
M\setminus D\left( R\right) .
\end{equation*}%
Since $u$ is harmonic on $F,$ the second inclusion of the lemma follows by
the maximum principle.
\end{proof}

As in (\ref{c16}) we have that 
\begin{equation}
\mathrm{Vol}\left( \left( E\setminus D\left( R\right) \right) \cap L\left(
\delta ,\infty \right) \right) \leq \frac{C}{\delta ^{2}}e^{-\frac{1}{C}R}
\label{V}
\end{equation}%
for any $\delta >0,$ where $C$ is independent of $R$ and $\delta.$

Denote 
\begin{equation*}
\widetilde{L}\left( t,\infty \right) =\left( \left\{ u>t\right\} \setminus
L\left( t,\infty \right) \right) \cap \left( L\left( \delta ,\infty \right)
\right)
\end{equation*}%
and 
\begin{equation}
\widetilde{l}\left( t\right) =\partial \widetilde{L}\left( t,\infty \right) .
\label{a8}
\end{equation}%
Since $u=f$ on $F=L\left(t_0, \infty \right)$ and $u$ has compact support on 
$E,$ it follows that $\widetilde{L}\left( t,\infty\right) $ is bounded.
Moreover, 
\begin{equation}
L\left( t,\infty \right) \subset F\text{ \ and \ }\widetilde{L}\left(
t,\infty \right) =\varnothing \text{ \ for }t\geq t_{0}.  \label{a9}
\end{equation}%
Note that Lemma \ref{area} and (\ref{a1}) imply that 
\begin{equation}
\int_{\delta }^{\infty }\mathrm{Area}\left( \widetilde{l}\left( t\right)
\right) dt\leq \frac{C}{\sqrt{R\delta ^{3}}}  \label{A}
\end{equation}%
for a constant $C$ independent of $R$ and $\delta,$ and Lemma \ref{I} also
holds with $C_{0}=1$ in view of the normalization (\ref{a3}).

\begin{lemma}
\label{I'}For any $\delta >0$ and any $a,b>\delta $, we have 
\begin{equation*}
\left\vert \int_{L\left( a,b\right) }\left\vert \nabla u\right\vert
^{2}\varphi \left( u\right) -\int_{a}^{b}\varphi \left( t\right)
dt\right\vert \leq \frac{C}{\sqrt{R}\delta ^{2}}\sup_{\left( a^{\prime
},b^{\prime }\right) }\left\vert \varphi \right\vert
\end{equation*}%
for any \thinspace integrable function $\varphi $ on $\left( a,b\right),$
where $a^{\prime }=\min \left\{ t_{0},a\right\} $ and $b^{\prime }=\min
\left\{ t_{0},b\right\},$ and the constant $C$ is independent of $R$ and $%
\delta.$
\end{lemma}

We also note the following.

\begin{lemma}
\label{Int_u}For any $\delta >0$ and any Lipschitz function $0\leq \chi \leq
1$ with compact support in $\left( \delta ,\infty \right) ,$ 
\begin{eqnarray*}
\lambda _{1}\left( M\right) \int_{L\left( \delta ,\infty \right) }u\chi
^{2}\left( u\right) &\leq &\frac{1}{4}\int_{\delta }^{\infty }t^{-1}\chi
^{2}\left( t\right) dt+\left( 1+\frac{C}{\sqrt{R}\delta ^{2}}\right)
\int_{\delta }^{\infty }t\left( \chi ^{\prime }\right) ^{2}\left( t\right) dt
\\
&&+\frac{C}{\sqrt{R}\delta ^{3}},
\end{eqnarray*}%
where $C$ is independent of $R$ and $\delta .$
\end{lemma}

\begin{proof}
Since $\chi (u)$ has compact support in $L\left( \delta ,\infty \right),$ 
\begin{eqnarray*}
\lambda _{1}\left( M\right) \int_{L\left( \delta ,\infty \right)}\left( u^{%
\frac{1}{2}}\chi \right) ^{2} &\leq &\int_{L\left( \delta ,\infty
\right)}\left\vert \nabla \left( u^{\frac{1}{2}}\chi \right) \right\vert ^{2}
\\
&=&\frac{1}{4}\int_{L\left( \delta ,\infty \right)}u^{-1}\left\vert \nabla
u\right\vert ^{2}\chi ^{2}+\frac{1}{2}\int_{L\left( \delta ,\infty
\right)}\left\langle \nabla u,\nabla \chi ^{2}\right\rangle \\
&&+\int_{L\left( \delta ,\infty \right)}u\left\vert \nabla \chi \right\vert
^{2}.
\end{eqnarray*}%
By Lemma \ref{I'}, 
\begin{equation*}
\int_{L\left( \delta ,\infty \right)}u^{-1}\left\vert \nabla u\right\vert
^{2}\chi ^{2} \leq \int_{\delta }^{\infty }t^{-1}\chi ^{2}\left( t\right) dt+%
\frac{C}{\sqrt{R}\delta ^{3}}
\end{equation*}
and 
\begin{equation*}
\int_{L\left( \delta ,\infty \right)}u\left\vert \nabla \chi \right\vert
^{2} \leq \left( 1+\frac{C}{\sqrt{R}\delta ^{2}}\right) \int_{\delta
}^{\infty }t\left( \chi ^{\prime }\right) ^{2}\left( t\right) dt.
\end{equation*}%
Also, by (\ref{a5}) and (\ref{V}), 
\begin{equation*}
\int_{L\left( \delta ,\infty \right)}\left\langle \nabla u,\nabla \chi
^{2}\right\rangle =-\int_{L\left( \delta ,\infty \right)}\chi ^{2}\Delta
u\leq \frac{C}{R\delta ^{2}}.
\end{equation*}%
The desired result follows by combining these estimates.
\end{proof}

We are now ready to prove Theorem \ref{A2}. By scaling, we may assume $K=1.$

\begin{theorem}
Let $\left( M,g\right) $ be a complete three-dimensional manifold with
scalar curvature bounded by $S\geq -6$ and bottom spectrum $\lambda
_{1}\left( M\right) =1.$ Assume that $H_{2}\left( M,\mathbb{Z}\right) $ has
no spherical classes and that the Ricci curvature of $M$ is bounded from
below by a constant. Then either all the ends of $M$ have infinite volume or 
$M=\mathbb{R}\times \mathbb{T}^{2}$ equipped with the warped product metric 
\begin{equation*}
ds_{M}^{2}=dt^{2}+e^{2t}ds_{\mathbb{T}^{2}}^{2}.
\end{equation*}
\end{theorem}

\begin{proof}
We apply the previous discussions to $\delta =\varepsilon ^{2}$ for any
given $\varepsilon >0.$ In the following, $R$ is arbitrarily large and
satisfies 
\begin{equation}
R>\frac{1}{\varepsilon ^{20}}.  \label{R}
\end{equation}%
Denote with 
\begin{equation}
\Omega \left( R\right) =F\cup E\left( R\right) .  \label{omega}
\end{equation}

Let $\phi =\phi \left( t\right) $ be a Lipschitz function such that $0\leq
\phi \leq 1$ and 
\begin{eqnarray}
\text{\textrm{supp}}\left( \phi \right) &\subset &\left( \varepsilon ^{2},%
\frac{1}{\varepsilon ^{2}}\right)  \label{a12} \\
\left\vert \phi ^{\prime }\right\vert &\leq &\frac{C}{\varepsilon ^{2}}. 
\notag
\end{eqnarray}%
Define $\phi \left( x\right) =\phi \left( u\left( x\right) \right) $ on $%
L\left( \varepsilon^{2},\varepsilon ^{-2}\right).$ From here on, we will
denote by $C$ a constant independent of $\varepsilon $ and $R$, whose value
may change from line to line.

Again, starting from the Bochner formula, we have 
\begin{eqnarray}
\int_{M}\left( \Delta \left\vert \nabla u\right\vert \right) \phi ^{2} &\geq
&\int_{M}\left( \left\vert u_{ij}\right\vert ^{2}-\left\vert \nabla
\left\vert \nabla u\right\vert \right\vert ^{2}+\mathrm{Ric}\left( \nabla
u,\nabla u\right) \right) \left\vert \nabla u\right\vert ^{-1}\phi ^{2}
\label{a13} \\
&&-\int_{M}\left\vert \nabla \left( \Delta u\right) \right\vert \phi ^{2}. 
\notag
\end{eqnarray}%
Using (\ref{c21}) that 
\begin{equation}
\int_{M}\left\vert \nabla \left( \Delta u\right) \right\vert \phi ^{2}\leq 
\frac{C}{R\varepsilon ^{4}}\leq C\varepsilon   \label{a14}
\end{equation}%
and that 
\begin{equation}
\int_{l\left( t\right) \cap \overline{M_{t}}}S_{t}\leq 0  \label{S}
\end{equation}%
on each regular level $l\left( t\right) $ by Lemma \ref{Level} and the
Gauss-Bonnet theorem, after manipulating as in the proof of Theorem \ref{A1}%
, we arrive at%
\begin{eqnarray}
&&\int_{M}\left( \left\vert u_{ij}\right\vert ^{2}-\left\vert \nabla
\left\vert \nabla u\right\vert \right\vert ^{2}+\mathrm{Ric}\left( \nabla
u,\nabla u\right) \right) \left\vert \nabla u\right\vert ^{-1}\phi ^{2}
\label{a18} \\
&\geq &\frac{3}{4}\int_{\varepsilon ^{2}}^{\varepsilon ^{-2}}\phi ^{2}\left(
t\right) \int_{l\left( t\right) \cap \Omega \left( R\right) }\left\vert
\nabla \left\vert \nabla u\right\vert \right\vert ^{2}\left\vert \nabla
u\right\vert ^{-2}dt-3\int_{M}\left\vert \nabla u\right\vert \phi ^{2} 
\notag \\
&&+\frac{1}{2}\int_{\varepsilon ^{2}}^{\varepsilon ^{-2}}\phi ^{2}\left(
t\right) \int_{\widetilde{l}\left( t\right) \cap \Omega \left( R\right)
}\left\vert \nabla \left\vert \nabla u\right\vert \right\vert ^{2}\left\vert
\nabla u\right\vert ^{-2}dt-C\varepsilon +\Gamma ,  \notag
\end{eqnarray}%
where 
\begin{equation}
\Gamma =\frac{1}{2}\int_{\Omega (R)}\left( \left\vert f_{ij}\right\vert ^{2}-%
\frac{3}{2}\left\vert \nabla \left\vert \nabla f\right\vert \right\vert
^{2}\right) \phi ^{2}+\frac{1}{2}\int_{\varepsilon ^{2}}^{\varepsilon
^{-2}}\phi ^{2}(t)\int_{l(t)\cap \Omega (R)}(S+6)dt  \label{Gamma'}
\end{equation}%
collects the relevant nonnegative terms discarded there.

Plugging (\ref{a18}) into (\ref{a13}) yields%
\begin{eqnarray}
&&-\int_{M}\left\langle \nabla \left\vert \nabla u\right\vert ,\nabla \phi
^{2}\right\rangle   \label{a19} \\
&\geq &\frac{3}{4}\int_{\varepsilon ^{2}}^{\varepsilon ^{-2}}\phi ^{2}\left(
t\right) \int_{l\left( t\right) \cap \Omega \left( R\right) }\left\vert
\nabla \left\vert \nabla u\right\vert \right\vert ^{2}\left\vert \nabla
u\right\vert ^{-2}dt  \notag \\
&&+\frac{1}{2}\int_{\varepsilon ^{2}}^{\varepsilon ^{-2}}\phi ^{2}\left(
t\right) \int_{\widetilde{l}\left( t\right) \cap \Omega \left( R\right)
}\left\vert \nabla \left\vert \nabla u\right\vert \right\vert ^{2}\left\vert
\nabla u\right\vert ^{-2}dt  \notag \\
&&-3\int_{M}\left\vert \nabla u\right\vert \phi ^{2}-C\varepsilon +\Gamma . 
\notag
\end{eqnarray}

We will make use of this inequality with different choices of $\phi $ below.
First, let $\phi =\chi \left( u\right) $ be a cut-off supported in $E,$
where 
\begin{equation*}
\chi \left( t\right) =\left\{ 
\begin{array}{c}
\frac{\ln t-2\ln \varepsilon }{\ln \left( \varepsilon ^{-1}\right) } \\ 
\frac{\ln \left( 2\varepsilon \right) -\ln t}{\ln 2}%
\end{array}%
\right. 
\begin{array}{c}
\text{for} \\ 
\text{for}%
\end{array}%
\begin{array}{c}
\varepsilon ^{2}\leq t\leq \varepsilon  \\ 
\varepsilon \leq t\leq 2\varepsilon .%
\end{array}%
\end{equation*}%
It follows that%
\begin{eqnarray}
&&\frac{3}{4}\int_{\varepsilon ^{2}}^{2\varepsilon }\chi ^{2}\left( t\right)
\int_{l\left( t\right) \cap E\left( R\right) }\left\vert \nabla \left\vert
\nabla u\right\vert \right\vert ^{2}\left\vert \nabla u\right\vert ^{-2}dt
\label{a20} \\
&&+\frac{1}{2}\int_{\varepsilon ^{2}}^{2\varepsilon }\chi ^{2}\left(
t\right) \int_{\widetilde{l}\left( t\right) \cap E\left( R\right)
}\left\vert \nabla \left\vert \nabla u\right\vert \right\vert ^{2}\left\vert
\nabla u\right\vert ^{-2}dt  \notag \\
&\leq &3\int_{M}\left\vert \nabla u\right\vert \chi
^{2}-\int_{M}\left\langle \nabla \left\vert \nabla u\right\vert ,\nabla \chi
^{2}\right\rangle +C\varepsilon .  \notag
\end{eqnarray}%
Following (\ref{c30}) to (\ref{gu}) we obtain%
\begin{equation*}
\int_{M}\left\vert \nabla u\right\vert ^{3}u^{-2}\chi ^{2}\leq
8\int_{M}u\chi ^{2}+C.
\end{equation*}%
By Lemma \ref{Int_u} we have 
\begin{eqnarray}
\int_{M}u\chi ^{2} &\leq &\frac{1}{4}\int_{\varepsilon ^{2}}^{2\varepsilon
}t^{-1}\chi ^{2}\left( t\right) dt+C\int_{\varepsilon ^{2}}^{2\varepsilon
}t\left( \chi ^{\prime }\right) ^{2}\left( t\right) dt+C  \label{chi} \\
&\leq &\frac{1}{12}\ln \left( \varepsilon ^{-1}\right) +C.  \notag
\end{eqnarray}%
Hence, this proves 
\begin{equation}
\int_{L\left( \varepsilon ^{2},\varepsilon \right) }\left\vert \nabla
u\right\vert ^{3}u^{-2}\chi ^{2}\leq \frac{2}{3}\ln \left( \varepsilon
^{-1}\right) +C.  \label{a24}
\end{equation}%
Conversely, applying $ab\leq \frac{a^{3}}{3}+\frac{2b^{\frac{3}{2}}}{3}$ for 
$a=2u^{\frac{1}{3}}$ and $b=u^{-\frac{4}{3}}\left\vert \nabla u\right\vert
^{2}$ yields 
\begin{equation}
2\int_{M}u^{-1}\left\vert \nabla u\right\vert ^{2}\chi ^{2}\leq \frac{8}{3}%
\int_{M}u\chi ^{2}+\frac{2}{3}\int_{M}u^{-2}\left\vert \nabla u\right\vert
^{3}\chi ^{2}.  \label{a25}
\end{equation}
Lemma \ref{I'} implies that 
\begin{equation*}
2\int_{M}u^{-1}\left\vert \nabla u\right\vert ^{2}\chi ^{2}\geq \frac{2}{3}%
\ln \left( \varepsilon ^{-1}\right) -C.
\end{equation*}%
Plugging this and (\ref{chi}) into (\ref{a25}) implies 
\begin{equation*}
\int_{M}u^{-2}\left\vert \nabla u\right\vert ^{3}\chi ^{2}\geq \frac{2}{3}%
\ln \left( \varepsilon ^{-1}\right) -C.
\end{equation*}%
Since by Lemma \ref{I'} and (\ref{a5}) 
\begin{equation*}
\int_{L\left( \varepsilon ,2\varepsilon \right) }u^{-2}\left\vert \nabla
u\right\vert ^{3}\chi ^{2}\leq C\int_{L\left( \varepsilon ,2\varepsilon
\right) }u^{-1}\left\vert \nabla u\right\vert ^{2}\leq C,
\end{equation*}%
we conclude 
\begin{equation*}
\int_{L\left( \varepsilon ^{2},\varepsilon \right) }u^{-2}\left\vert \nabla
u\right\vert ^{3}\chi ^{2}\geq \frac{2}{3}\ln \left( \varepsilon
^{-1}\right) -C.
\end{equation*}%
Together with (\ref{a24}), this proves 
\begin{equation}
\left\vert \int_{L\left( \varepsilon ^{2},\varepsilon \right)
}u^{-2}\left\vert \nabla u\right\vert ^{3}\chi ^{2}-\frac{2}{3}\ln \left(
\varepsilon ^{-1}\right) \right\vert \leq C.  \label{I1}
\end{equation}

A similar argument on the end $F$ implies 
\begin{equation}
\left\vert \int_{L\left( \varepsilon ^{-1},\varepsilon ^{-2}\right)
}u^{-2}\left\vert \nabla u\right\vert ^{3}\eta ^{2}-\frac{2}{3}\ln \left(
\varepsilon ^{-1}\right) \right\vert \leq C,  \label{I2}
\end{equation}%
where 
\begin{equation*}
\eta \left( t\right) =\left\{ 
\begin{array}{c}
\frac{\ln t-\ln \left( 2\varepsilon \right) ^{-1}}{\ln 2} \\ 
\frac{\ln \left( \varepsilon ^{-2}\right) -\ln t}{\ln \left( \varepsilon
^{-1}\right) }%
\end{array}%
\right. 
\begin{array}{c}
\text{for} \\ 
\text{for}%
\end{array}%
\begin{array}{c}
\frac{1}{2}\varepsilon ^{-1}\leq t\leq \varepsilon ^{-1} \\ 
\varepsilon ^{-1}\leq t\leq \varepsilon ^{-2}.%
\end{array}%
\end{equation*}%
We now return to (\ref{a19}) and replace $\phi $ by $\phi ^{\frac{3}{2}}$ to
obtain 
\begin{eqnarray}
&&-\int_{M}\left\langle \nabla \left\vert \nabla u\right\vert ,\nabla \phi
^{3}\right\rangle   \label{a27} \\
&\geq &\frac{3}{4}\int_{\varepsilon ^{2}}^{\varepsilon ^{-2}}\phi ^{3}\left(
t\right) \int_{l\left( t\right) \cap \Omega \left( R\right) }\left\vert
\nabla \left\vert \nabla u\right\vert \right\vert ^{2}\left\vert \nabla
u\right\vert ^{-2}dt  \notag \\
&&+\frac{1}{2}\int_{\varepsilon ^{2}}^{\varepsilon ^{-2}}\phi ^{3}\left(
t\right) \int_{\widetilde{l}\left( t\right) \cap \Omega \left( R\right)
}\left\vert \nabla \left\vert \nabla u\right\vert \right\vert ^{2}\left\vert
\nabla u\right\vert ^{-2}dt  \notag \\
&&-3\int_{M}\left\vert \nabla u\right\vert \phi ^{3}-C\varepsilon +\Gamma . 
\notag
\end{eqnarray}%
Let us now set 
\begin{equation*}
\phi \left( t\right) =\left\{ 
\begin{array}{c}
\frac{\ln t-2\ln \varepsilon }{\ln \left( \varepsilon ^{-1}\right) } \\ 
1 \\ 
\frac{2\ln \left( \varepsilon ^{-1}\right) -\ln t}{\ln \left( \varepsilon
^{-1}\right) }%
\end{array}%
\right. 
\begin{array}{c}
\text{for} \\ 
\text{for} \\ 
\text{for}%
\end{array}%
\begin{array}{c}
\varepsilon ^{2}\leq t\leq \varepsilon  \\ 
\varepsilon \leq t\leq \varepsilon ^{-1} \\ 
\varepsilon ^{-1}\leq t\leq \varepsilon ^{-2}.%
\end{array}%
\end{equation*}%
We claim that 
\begin{equation}
-\int_{M}\left\langle \nabla \left\vert \nabla u\right\vert ,\nabla \phi
^{3}\right\rangle \leq \frac{C}{\ln \left( \varepsilon ^{-1}\right) }.
\label{a28}
\end{equation}%
Indeed, we have 
\begin{eqnarray}
-\int_{M}\left\langle \nabla \left\vert \nabla u\right\vert ,\nabla \phi
^{3}\right\rangle  &=&-\frac{3}{\ln \left( \varepsilon ^{-1}\right) }%
\int_{L\left( \varepsilon ^{2},\varepsilon \right) }\left\langle \nabla
\left\vert \nabla u\right\vert ,\nabla u\right\rangle u^{-1}\phi ^{2}
\label{a30} \\
&&+\frac{3}{\ln \left( \varepsilon ^{-1}\right) }\int_{L\left( \varepsilon
^{-1},\varepsilon ^{-2}\right) }\left\langle \nabla \left\vert \nabla
u\right\vert ,\nabla u\right\rangle u^{-1}\phi ^{2}.  \notag
\end{eqnarray}%
On the one hand, integration by parts gives 
\begin{eqnarray*}
-\int_{L\left( \varepsilon ^{2},\varepsilon \right) }\left\langle \nabla
\left\vert \nabla u\right\vert ,\nabla u\right\rangle u^{-1}\phi ^{2}
&=&\int_{L\left( \varepsilon ^{2},\varepsilon \right) }u^{-1}\left\vert
\nabla u\right\vert \left( \Delta u\right) \phi ^{2}-\int_{L\left(
\varepsilon ^{2},\varepsilon \right) }u^{-2}\left\vert \nabla u\right\vert
^{3}\phi ^{2} \\
&&+\int_{L\left( \varepsilon ^{2},\varepsilon \right) }\left\vert \nabla
u\right\vert \left\langle \nabla u,\nabla \phi ^{2}\right\rangle
u^{-1}-\int_{l\left( \varepsilon \right) }\left\vert \nabla u\right\vert
^{2}u^{-1} \\
&\leq &-\frac{2}{3}\ln \left( \varepsilon ^{-1}\right) +C,
\end{eqnarray*}%
where in the last line we have used (\ref{I1}) and Lemma \ref{I'}. On the
other hand, integration by parts as above together with (\ref{I2}) implies 
\begin{equation*}
\int_{L\left( \varepsilon ^{-1},\varepsilon ^{-2}\right) }\left\langle
\nabla \left\vert \nabla u\right\vert ,\nabla u\right\rangle u^{-1}\phi
^{2}\leq \frac{2}{3}\ln \left( \varepsilon ^{-1}\right) +C.
\end{equation*}%
Plugging these estimates into (\ref{a30}) we obtain (\ref{a28}). Hence, by (%
\ref{a27}), we have established that 
\begin{eqnarray}
&&\Gamma +\frac{3}{4}\int_{\varepsilon ^{2}}^{\varepsilon ^{-2}}\phi
^{3}\left( t\right) \int_{l\left( t\right) \cap \Omega \left( R\right)
}\left\vert \nabla \left\vert \nabla u\right\vert \right\vert ^{2}\left\vert
\nabla u\right\vert ^{-2}dt  \label{a31} \\
&&+\frac{1}{2}\int_{\varepsilon ^{2}}^{\varepsilon ^{-2}}\phi ^{3}\left(
t\right) \int_{\widetilde{l}\left( t\right) \cap \Omega \left( R\right)
}\left\vert \nabla \left\vert \nabla u\right\vert \right\vert ^{2}\left\vert
\nabla u\right\vert ^{-2}dt  \notag \\
&\leq &3\int_{M}\left\vert \nabla u\right\vert \phi ^{3}+\frac{C}{\ln \left(
\varepsilon ^{-1}\right) }.  \notag
\end{eqnarray}%
Proceeding as in the proof of (\ref{c31}) we note that 
\begin{eqnarray*}
\int_{M}\left\langle \nabla \left\vert \nabla u\right\vert ,\nabla
u\right\rangle u^{-1}\phi ^{3} &=&-\int_{M}\left\vert \nabla u\right\vert
u^{-1}\left( \Delta u\right) \phi ^{3}+\int_{M}\left\vert \nabla
u\right\vert ^{3}u^{-2}\phi ^{3} \\
&&-\int_{M}\left\vert \nabla u\right\vert u^{-1}\left\langle \nabla u,\nabla
\phi ^{3}\right\rangle .
\end{eqnarray*}%
However, (\ref{I1}) and (\ref{I2}) imply that 
\begin{equation*}
\left\vert \int_{M}\left\vert \nabla u\right\vert u^{-1}\left\langle \nabla
u,\nabla \phi ^{3}\right\rangle \right\vert \leq \frac{C}{\ln \left(
\varepsilon ^{-1}\right) }.
\end{equation*}%
Also, 
\begin{equation*}
\left\vert -\int_{M}\left\vert \nabla u\right\vert u^{-1}\left( \Delta
u\right) \phi ^{3}\right\vert \leq C\,\varepsilon 
\end{equation*}%
by (\ref{a5}) and (\ref{V}). Therefore, 
\begin{equation*}
\int_{M}\left\vert \nabla u\right\vert ^{3}u^{-2}\phi ^{3}\leq \int_{\Omega
\left( R\right) }\left\langle \nabla \left\vert \nabla u\right\vert ,\nabla
u\right\rangle u^{-1}\phi ^{3}+\frac{C}{\ln \left( \varepsilon ^{-1}\right) }%
.
\end{equation*}%
Now arguing as from (\ref{c31}) to (\ref{gu}), we conclude from (\ref{a31})
and the above inequality that 
\begin{equation}
\,\Gamma +\int_{M}\left\vert \nabla u\right\vert ^{3}u^{-2}\phi ^{3}\leq
8\int_{M}u\phi ^{3}+\frac{C}{\ln \left( \varepsilon ^{-1}\right) }.
\label{a32}
\end{equation}%
On the other hand, as $\lambda _{1}(M)=1,$ Lemma \ref{Int_u} implies that 
\begin{equation}
\int_{M}u\phi ^{3}\leq \frac{1}{4}\int_{M}\left\vert \nabla u\right\vert
^{2}u^{-1}\phi ^{3}+\frac{C}{\ln \left( \varepsilon ^{-1}\right) }.
\label{a33}
\end{equation}%
However, as in (\ref{a25}),%
\begin{equation}
2u^{-1}\left\vert \nabla u\right\vert ^{2}\leq \frac{2}{3}u^{-2}\left\vert
\nabla u\right\vert ^{3}+\frac{8}{3}u.  \label{a34}
\end{equation}%
Applying (\ref{a32}) and (\ref{a33}) it follows that  
\begin{eqnarray*}
0 &\leq &\int_{M}\left( \frac{2}{3}u^{-2}\left\vert \nabla u\right\vert ^{3}+%
\frac{8}{3}u-2u^{-1}\left\vert \nabla u\right\vert ^{2}\right) \phi ^{3} \\
&\leq &8\int_{M}u\phi ^{3}-2\int_{M}u^{-1}\left\vert \nabla u\right\vert
^{2}\phi ^{3}-\frac{2}{3}\Gamma +\frac{C}{\ln \left( \varepsilon
^{-1}\right) } \\
&\leq &-\frac{2}{3}\Gamma +\frac{C}{\ln \left( \varepsilon ^{-1}\right) }.
\end{eqnarray*}%
In conclusion, this proves that 
\begin{eqnarray}
&&\int_{L\left( \varepsilon ,\varepsilon ^{-1}\right) \cap \Omega (R)}\left( 
\frac{2}{3}f^{-2}\left\vert \nabla f\right\vert ^{3}+\frac{8}{3}%
f-2f^{-1}\left\vert \nabla f\right\vert ^{2}\right)   \label{a35} \\
&&+\frac{1}{3}\int_{L\left( \varepsilon ,\varepsilon ^{-1}\right) \cap
\Omega (R)}\left( \left\vert f_{ij}\right\vert ^{2}-\frac{3}{2}\left\vert
\nabla \left\vert \nabla f\right\vert \right\vert ^{2}\right)   \notag \\
&&+\frac{1}{3}\int_{\varepsilon }^{\varepsilon ^{-1}}\int_{l(t)\cap \Omega
(R)}(S+6)dt  \notag \\
&\leq &\frac{C}{\ln \left( \varepsilon ^{-1}\right) }.  \notag
\end{eqnarray}

Since $R$ and $\varepsilon $ are arbitrary, we conclude from here that $%
\left\vert f_{ij}\right\vert ^{2}=\frac{3}{2}\left\vert \nabla \left\vert
\nabla f\right\vert \right\vert ^{2}$ and $\left\vert \nabla f\right\vert =2f
$.  As in \cite{LW}, this proves that $\left( M,g\right) $ splits as a
warped product $M=\mathbb{R}\times N$ with $ds_{M}^{2}=dt^{2}+e^{2h\left(
t\right) }ds_{N}^{2}$, and $f\left( t\right) =Ce^{-2t}$ if we choose 
$F=\left( -\infty ,0\right) \times N$ to be the parabolic end. Since $f$ is harmonic, it follows $%
h\left( t\right) =t$. Moroever, as $S=-6$ by (\ref{a35}), it follows that $N$
is flat. The theorem is proved.
\end{proof}

\section{Splitting for nonnegative scalar curvature\label{Sge0}}

In this section we study complete three-dimensional manifolds of
non-negative scalar curvature with more than one end. We first prove Theorem %
\ref{A3}.

\begin{theorem}
Let $\left( M,g\right) $ be a complete three-dimensional parabolic manifold
with its scalar curvature $S\geq 0$ and Ricci curvature bounded below by a
constant. Assume that $H_{2}\left( M,\mathbb{Z}\right) $ contains no
spherical classes. Then either $M$ is connected at infinity or it splits
isometrically as a direct product $M=\mathbb{R}\times \mathbb{T}^2.$
\end{theorem}

\begin{proof}
Let us assume that $M$ has more than one end. As $M$ is parabolic, so is
each end. Then by \cite{N} and \cite{LT1} there exists a harmonic function $%
u:M\rightarrow \mathbb{R}$ such that 
\begin{equation*}
\lim_{x\rightarrow \infty, \ x\in E}u\left( x\right) =\infty \text{ \ and }%
\lim_{x\rightarrow \infty,\ x\in F}u\left( x\right) =-\infty ,
\end{equation*}%
where $E$ is an end of $M$ and $F=M\setminus E.$ Denote with 
\begin{equation*}
w\left( t\right) =\int_{l\left( t\right) }\left\vert \nabla u\right\vert
^{2},
\end{equation*}%
where $l\left( t\right) =\left\{ u=t\right\}.$ Denote also with $%
L(s,t)=\{s<u<t\}.$

Fix $s\in \mathbb{R}$ and let $\phi =\phi \left( t\right) $ be a $C^{2}$
function on $\left( s,\infty \right) $. Following \cite{MW}, we note that
for $t>s,$ 
\begin{eqnarray}
&&\left( \phi \left( t\right) w^{\prime }\left( t\right) -\phi ^{\prime
}\left( t\right) w\left( t\right) \right) -\left( \phi \left( s\right)
w^{\prime }(s)-\phi ^{\prime }\left( s\right) w\left( s\right) \right)
\label{b1} \\
&=&\int_{L\left( s,t\right) }\left( \phi \left( u\right) \Delta \left\vert
\nabla u\right\vert -\left\vert \nabla u\right\vert \Delta \phi \left(
u\right) \right)  \notag \\
&=&\int_{L\left( s,t\right) }\frac{\phi \left( u\right) }{\left\vert \nabla
u\right\vert }\left( \left\vert u_{ij}\right\vert ^{2}-\left\vert \nabla
\left\vert \nabla u\right\vert \right\vert ^{2}+\mathrm{Ric}\left( \nabla
u,\nabla u\right) \right)  \notag \\
&&-\int_{s}^{t}\phi ^{\prime \prime }\left( r\right) w\left( r\right) dr. 
\notag
\end{eqnarray}%
Since $H_{2}\left( M,\mathbb{Z}\right) $ contains no spherical classes, the
Gauss-Bonnet formula implies 
\begin{equation*}
\int_{l\left( r\right) }S_{r}\leq 0.
\end{equation*}
By Lemma \ref{RicS} and the fact that $S\geq 0,$ we get 
\begin{eqnarray*}
&&\int_{l\left( r\right) }\frac{1}{\left\vert \nabla u\right\vert ^{2}}%
\left( \left\vert u_{ij}\right\vert ^{2}-\left\vert \nabla \left\vert \nabla
u\right\vert \right\vert ^{2}+\mathrm{Ric}\left( \nabla u,\nabla u\right)
\right) \\
&=&\frac{1}{2}\int_{l\left( r\right) }\left( S-S_{r}+\frac{1}{\left\vert
\nabla u\right\vert ^{2}}\left\vert u_{ij}\right\vert ^{2}\right) \\
&\geq &\frac{1}{2}\int_{l\left( r\right) }\frac{1}{\left\vert \nabla
u\right\vert ^{2}}\left\vert u_{ij}\right\vert ^{2} \\
&\geq &\frac{3}{4}\int_{l\left( r\right) }\frac{1}{\left\vert \nabla
u\right\vert ^{2}}\left\vert \nabla \left\vert \nabla u\right\vert
\right\vert ^{2} \\
&\geq &\frac{3}{4}\frac{\left( w^{\prime }(r)\right) ^{2}}{w(r)},
\end{eqnarray*}%
where the last line follows from 
\begin{equation*}
|w^{\prime }|(r)\leq \int_{l(r)}|\nabla |\nabla u||\ \leq \left( \int_{l(r)}%
\frac{|\nabla |\nabla u||^{2}}{|\nabla u|^{2}}\right) ^{\frac{1}{2}}w(r)^{%
\frac{1}{2}}.
\end{equation*}%
Therefore, by (\ref{b1}) we conclude that 
\begin{eqnarray}
&&\left( \phi \left( t\right) w^{\prime }\left( t\right) -\phi ^{\prime
}\left( t\right) w\left( t\right) \right) -\left( \phi \left( s\right)
w^{\prime }(s)-\phi ^{\prime }\left( s\right) w\left( s\right) \right)
\label{ID} \\
&\geq &-\int_{s}^{t}\phi ^{\prime \prime }\left( r\right) w\left( r\right)
dr+\frac{3}{4}\int_{s}^{t}\phi \left( r\right) \frac{\left( w^{\prime
}\right) ^{2}}{w}\left( r\right) dr  \notag
\end{eqnarray}%
for any $C^{2}$ function $\phi \geq 0$ on $\left( s,\infty \right).$ So 
\begin{equation}
w^{\prime }(t)\geq w^{\prime }(s),  \label{b2}
\end{equation}%
for all $s<t$ by taking $\phi \left( r\right) =1.$ Letting $\phi
(r)=r^{-\alpha }$ for $\alpha >0$ and $r\geq 1,$ we obtain that 
\begin{eqnarray}
&&\left( t^{-\alpha }w^{\prime }\left( t\right) +\alpha t^{-\alpha
-1}w\left( t\right) \right) -\left( s^{-\alpha }w^{\prime }(s)+\alpha
s^{-\alpha -1}w\left( s\right) \right)  \label{b3} \\
&\geq &-\alpha (\alpha +1)\int_{s}^{t}r^{-\alpha -2}w\left( r\right) dr+%
\frac{3}{4}\int_{s}^{t}r^{-\alpha }\frac{\left( w^{\prime }\right) ^{2}}{w}%
\left( r\right) dr  \notag
\end{eqnarray}%
for all $1\leq s\leq t.$

We rewrite the elementary inequality 
\begin{equation*}
\frac{2\varepsilon }{r}w^{\prime }(r)\leq \frac{\left( w^{\prime }\right)
^{2}}{w}\left( r\right) +\frac{\varepsilon ^{2}}{r^{2}}w(r)
\end{equation*}%
into 
\begin{equation*}
\frac{3}{4}\int_{s}^{t}r^{-\alpha }\frac{\left( w^{\prime }\right) ^{2}}{w}%
\left( r\right) dr\geq \frac{3\varepsilon }{2}\int_{s}^{t}r^{-\alpha
-1}w^{\prime }(r)dr-\frac{3\varepsilon ^{2}}{4}\int_{s}^{t}r^{-\alpha
-2}w(r)dr.
\end{equation*}%
After integrating by parts the first term on the right side we get 
\begin{eqnarray*}
\frac{3}{4}\int_{s}^{t}r^{-\alpha }\frac{\left( w^{\prime }\right) ^{2}}{w}%
\left( r\right) dr &\geq &\frac{3\varepsilon }{2}\left( t^{-\alpha
-1}w(t)-s^{-\alpha -1}w(s)\right) \\
&+&\left( \frac{3\varepsilon }{2}(\alpha +1)-\frac{3\varepsilon ^{2}}{4}%
\right) \int_{s}^{t}r^{-\alpha -2}w(r)dr.
\end{eqnarray*}%
Plugging this inequality into (\ref{b3}) implies that 
\begin{eqnarray*}
&&\left( t^{-\alpha }w^{\prime }\left( t\right) -\left( \frac{3\varepsilon }{%
2}-\alpha \right) t^{-\alpha -1}w\left( t\right) \right) -\left( s^{-\alpha
}w^{\prime }(s)-\left( \frac{3\varepsilon }{2}-\alpha \right) s^{-\alpha
-1}w\left( s\right) \right) \\
&\geq &\left( \frac{3\varepsilon }{2}(\alpha +1)-\frac{3\varepsilon ^{2}}{4}%
-\alpha (\alpha +1)\right) \int_{s}^{t}r^{-\alpha -2}w\left( r\right) dr
\end{eqnarray*}%
for any $1\leq s\leq t.$ For $\alpha =\varepsilon =\frac{1}{4}$ it follows
that 
\begin{equation}
t^{-\frac{1}{4}}w^{\prime }\left( t\right) -{\frac{1}{8}}t^{-\frac{5}{4}%
}w\left( t\right) \geq s^{-\frac{1}{4}}w^{\prime }(s)-{\frac{1}{8}}s^{-\frac{%
5}{4}}w\left( s\right) .  \label{b5}
\end{equation}

As the Ricci curvature of $M$ is bounded from below by a constant, the
gradient estimate (\ref{CY}) implies 
\begin{equation*}
w(t)=\int_{l(t)}|\nabla u|^{2}\leq Ct\int_{l(t)}|\nabla u|=Ct
\end{equation*}%
for all $t\geq 1.$ In particular, there exists $t_{n}\rightarrow \infty $
such that $w^{\prime }(t_{n})\leq C.$ Letting $t=t_{n}$ in (\ref{b5}) and
making $n\rightarrow \infty $ we conclude that 
\begin{equation*}
s^{-\frac{1}{4}}w^{\prime }(s)\leq {\frac{1}{8}}s^{-\frac{5}{4}}w\left(
s\right) 
\end{equation*}%
for all $s\geq 1.$ Integrating this differential inequality from $s=1$ to $%
s=r\geq 1$ we conclude that 
\begin{equation}
w(r)\leq Cr^{\frac{1}{8}}  \label{b6}
\end{equation}%
for all $r\geq 1.$ A similar argument implies that 
\begin{equation*}
w(r)\leq C|r|^{\frac{1}{8}}
\end{equation*}%
for all $r\leq -1$ as well. Since $w$ is convex by (\ref{b2}), this
immediately implies that $w$ is a constant function. Therefore, (\ref{b2})
becomes an equality. Now tracing through the derivation of (\ref{b2}), one
sees that $u_{ij}=0$ and $S=0$. This proves the splitting.
\end{proof}

Recall that a nonparabolic manifold is called regular if the Green's
function $G\left( p,x\right) $ satisfies 
\begin{equation*}
\lim_{x\rightarrow \infty }G\left( p,x\right) =0.
\end{equation*}%
Similarly, a nonparabolic end $E$ of $M$ is called regular if $G$ is regular
on $E.$

\begin{theorem}
Let $\left( M,g\right) $ be a nonparabolic complete three-dimensional
manifold with $S\geq 0$ and Ricci curvature bounded below. Assume that $%
H_{2}\left( M,\mathbb{Z}\right) $ contains no spherical classes. If all
nonparabolic ends of $M$ are regular, then $M$ has one end.
\end{theorem}

\begin{proof}
Suppose first that $M$ has at least one parabolic end. We let $F$ be the
union of all parabolic ends and $E=M\setminus F$. Then $E$ is nonparabolic.
Again, by \cite{N, LT1}, there exists a harmonic function $u:M\rightarrow
\left( 0,\infty \right) $ such that 
\begin{eqnarray*}
\lim_{x\rightarrow \infty ,x\in E}u\left( x\right)  &=&0 \\
\lim_{x\rightarrow \infty ,x\in F}u\left( x\right)  &=&\infty .
\end{eqnarray*}%
According to (\ref{b2}) we have $w^{\prime }\left( r\right) \geq w^{\prime
}\left( s\right) $ for all $r>s>0,$ whereas (\ref{b6}) implies $w^{\prime
}\left( r_{n}\right) \rightarrow 0$ for some sequence $r_{n}\rightarrow
\infty .$ Hence, $w^{\prime }\left( s\right) \leq 0$ for all $s>0.$
Therefore, $w\left( r\right) \leq w\left( \varepsilon \right) $ for $%
r>\varepsilon .$ Letting $\varepsilon \rightarrow 0$ implies $w=0.$ This
contradiction shows that all ends of $M$ must be nonparabolic.

Assume now that $M$ has a nonparabolic end $E$ and let $F=M\setminus E.$
Since both $E$ and $F$ are regular, there exists a harmonic function $%
u:M\rightarrow \left( 0,1\right) $ such that 
\begin{equation*}
\lim_{x\rightarrow \infty ,x\in E}u\left( x\right) =0\text{ \ and }%
\lim_{x\rightarrow \infty ,x\in F}u\left( x\right) =1.
\end{equation*}%
For $\phi \left( r\right) =r$ in (\ref{ID}) we obtain 
\begin{equation}
tw^{\prime }\left( t\right) -w\left( t\right) \geq sw^{\prime }\left(
s\right) -w\left( s\right) .  \label{b7}
\end{equation}%
However, the gradient estimate implies that $0<w\left( s\right) \leq Cs.$ In
particular, there exists a sequence $s_{n}\rightarrow 0$ such that 
\begin{equation*}
\liminf_{n \rightarrow \infty} w^{\prime }\left( s_{n}\right) \geq 0.
\end{equation*}%
Hence, (\ref{b7}) implies that $tw^{\prime }\left( t\right) -w\left(
t\right) \geq 0$ for all $t\in \left( 0,1\right) .$ Integrating in $t$ we
get 
\begin{equation}
\frac{w\left( t\right) }{t}\geq \frac{w\left( s\right) }{s}  \label{b8}
\end{equation}%
for all $0<s<t<1.$ However, the gradient estimate implies that $w\left(
t\right) \leq C\left( 1-t\right) $ for $t<1.$ After letting $t\rightarrow 1$
in (\ref{b8}) we conclude that $w=0.$ This contradiction shows that $M$ has
only one end.
\end{proof}

\end{document}